\documentclass[final,hidelinks,onefignum,onetabnum]{siamart251216}

\usepackage[margin=1in]{geometry}
\usepackage{mathrsfs,amsmath,amssymb,amsfonts,amsopn}
\usepackage{graphicx}
\usepackage{enumitem}
\usepackage{hyperref}

\ifpdf
  \DeclareGraphicsExtensions{.eps,.pdf,.png,.jpg}
\else
  \DeclareGraphicsExtensions{.eps}
\fi

\newcommand{\R}{\mathbb{R}}
\newcommand{\tu}[1]{\mathrm{#1}} 
\newcommand{\C}{\mathbb{C}}
\newcommand{\Z}{\mathbb{Z}}
\newcommand{\smat}[1]{\left[\begin{smallmatrix}#1\end{smallmatrix}\right]}
\newcommand{\bmat}[1]{\begin{bmatrix}#1\end{bmatrix}}
\newcommand{\smatNoB}[1]{\begin{smallmatrix}#1\end{smallmatrix}}

\newcommand{\Obs}{\mathcal{O}}
\newcommand{\Tu}{\mathcal{T}_{\tu{u}}}
\newcommand{\Tw}{\mathcal{T}_{\tu{w}}}
\newcommand{\Ru}{\mathcal{R}_{\tu{u}}}
\newcommand{\Rw}{\mathcal{R}_{\tu{w}}}
\newcommand{\dotsS}{\rotatebox{0}{\tiny\dots}}
\newcommand{\vdotsS}{\rotatebox{90}{\tiny\dots}}
\newcommand{\ddotsS}{\rotatebox{135}{\tiny\dots}}
\newcommand{\e}{\mathrm{e}} 
\newcommand{\newM}{\mathcal{M}} 
\newcommand{\Ab}{\mathbf{A}}
\newcommand{\Bb}{\mathbf{B}}
\newcommand{\Lb}{\mathbf{L}}
\newcommand{\Fb}{\mathbf{F}}
\newcommand{\Pb}{\mathbf{P}}
\newcommand{\Sb}{\mathbf{S}}
\renewcommand{\i}{\mathrm{i}} 
\newenvironment{smallarray}[1]
 {\null\,\vcenter\bgroup\scriptsize
  \arraycolsep=.13885em
  \hbox\bgroup$\array{@{}#1@{}}}
 {\endarray$\egroup\egroup\,\null}

\newsiamremark{remark}{Remark}
\newsiamremark{hypothesis}{Hypothesis}
\crefname{hypothesis}{Hypothesis}{Hypotheses}
\newsiamthm{claim}{Claim}
\newsiamremark{fact}{Fact}
\crefname{fact}{Fact}{Facts}
\newsiamremark{assumption}{Assumption}
\newsiamremark{problem}{Problem}

\headers{Output regulation via input-output data}{A. Bisoffi, W. Liu, Z. Hu, and C. De Persis}
\title{Output regulation via input-output data%
\thanks{
The work of Wenjie Liu was carried out while she was affiliated with the Engineering and Technology Institute, University of Groningen.
}
}

\author{
Andrea Bisoffi\footnotemark[2]\thanks{Department of Electronics, Information, and Bioengineering, Politecnico di Milano, 20133, Italy (\email{andrea.bisoffi@polimi.it}).}
\and 
Wenjie Liu\thanks{School of Electrical and Electronic Engineering, Nanyang Technological University, Singapore 639798, Singapore (\email{wenjie.liu@ntu.edu.sg}).}
\and 
Zhongjie Hu\thanks{School of Artificial Intelligence and Automation and Hubei Key Laboratory of Brain-Inspired Intelligent Systems, Huazhong University of Science and Technology, Wuhan, Hubei 430074, China (\email{zhongjiehu@hust.edu.cn}).}
\and 
Claudio \mbox{De Persis}\thanks{Engineering and Technology Institute, University of Groningen, 9747 AG, The Netherlands (\email{c.de.persis@rug.nl}).}}

\DeclareMathOperator{\rank}{rank}

\ifpdf
\hypersetup{
  pdftitle={Output regulation via input-output data},
  pdfauthor={A. Bisoffi, W. Liu, Z. Hu, C. De Persis}
}
\fi

\allowdisplaybreaks

\begin{document}

\maketitle

\begin{abstract}
From a multi-input-multi-output (MIMO) discrete-time linear system, we collect input-output data affected by noise in the form of an unknown exosignal and, from these data points (without knowledge of the system model), we design a feedback controller that asymptotically annihilates the effect of that exosignal on the output.
This amounts to solving an output regulation problem purely from input-output data, for MIMO linear systems.
The design of the controller corresponds to a semidefinite program and is pursued on a suitable auxiliary system.
Such design carries over from the auxiliary system to the original one by a rigorous examination of the relation between the solutions of the two systems.
\end{abstract}


\section{Introduction}

The objective of output regulation is the design of a controller that asymptotically annihilates the effect of a so-called exosignal on selected error variables while ensuring asymptotic stabilization whenever the exosignal is set to zero.
The exosignal corresponds, in general, to reference or disturbance signals and belongs to canonical classes (e.g., step, ramp, sinusoidal or exponential signals) since it is generated by a dynamical system called the exosystem.
The exosignal acts both when data points are collected and when a controller is put in feedback with the system to be controlled.
We would like to address the problem of designing a controller for output regulation using only data points of the control inputs and of the outputs.

Within data-based output regulation, a first significant challenge comes from handling the unknown exosignal contained in the collected data.
For instance, the exosignal is measured during data collection in~\cite{Zhu2024output-regulation} and \cite{li2026cooperative}.
However, when the exosignal corresponds to a disturbance, measuring it becomes impractical. 
To the best of our knowledge this challenge was first addressed in \cite{hu2024dd-output-regulation}, see also \cite{hu2025enforcing}, via the framework in~\cite{de2019formulas} and for continuous-time systems.
A second significant challenge comes when, for multi-input-multi-output (MIMO) systems, input-output data are used instead of input-state data.
In applications, availability of the former is much more common than that of the latter.
For instance, \cite{Zhu2024output-regulation} uses input-state data (and state feedback) in a discrete-time setting and \cite{hu2024dd-output-regulation} uses input-state-state derivative data in a continuous-time setting.
Here, we would like to design a controller for output regulation from data while overcoming both previous challenges.

\subsection{Contribution} The defining features of the proposed approach are the following. (i)~The influence of the unknown exosignal on input-output data is completely removed. 
(ii)~As a consequence of knowing the exosystem, no bounds on the ``size'' of the exosignal are required to design a controller.
(iii)~The design of the output regulation controller is based on an auxiliary system associated with the actual system to be controlled and a rigorous relation between the two is established in terms of their solutions.
(iv)~Controller design corresponds to a semidefinite, and thus convex, program.

\subsection{Relation with the literature}

The present treatment builds upon our previous works in \cite{de2019formulas,li2024controller,hu2024dd-output-regulation}. From~\cite{de2019formulas,li2024controller}, we borrow the approach of associating an auxiliary system with the actual system in order to deal with input-output data.
On the other hand, \cite{de2019formulas} uses noise-free data and \cite{li2024controller} uses input-output data corrupted by bounded noise and both solve a stabilization problem, which is a special case of an output regulation one.
Here, no bounds on the noise are assumed and the approach borrowed from ~\cite{de2019formulas,li2024controller} requires to be adapted to incorporate the presence of the exosignal during data collection and control execution. Dispensing with noise bounds is enabled by a technique from~\cite{hu2024dd-output-regulation}.
\cite{hu2024dd-output-regulation} also solves a regulation problem, but uses input-state-state derivative data.
This is a key difference with respect to here, where the use of input-output data induces nontrivial sophistication in the setting.
Additional differences are that \cite{hu2024dd-output-regulation} operates in continuous time whereas we develop the present approach in discrete time and that the present assumption on the exosystem, see Assumption~\ref{ass:S}, is less stringent than \cite[Assumption~2]{hu2024dd-output-regulation} and allows the consideration of a larger set of exosignals. 

Other works focused on data-based regulation are \cite{li2026cooperative,liu2025data,liuguo2025data,bosso2025data,mao2025one}, which we now discuss one by one.
\cite{li2026cooperative} assumes that the exosignal is measured during data collection, which does not capture the realistic case when the exosignal is a disturbance, and addresses a multiagent system whereas we focus here on the problem of output regulation from input-output data and also on some of current limitations, see Proposition~\ref{proposition:ex_remark}.
\cite{liu2025data} works in a continuous-time setting and requires measurements of  state and state derivatives. Furthermore, as it does not employ the disturbance filtering technique of \cite{hu2024dd-output-regulation}, the exosignal must be assumed to be bounded by known constants.
As for \cite{liuguo2025data}, its approach is different in nature from ours since it requires a preliminary incremental passivation; \cite{liuguo2025data} also considers a continuous-time setting and relies on state and state derivatives.
\cite{bosso2025data} considers the design of a controller for output regulation based on input-output data, but while  it provides results based on the arguments of \cite{de2019formulas,li2024controller} on one hand, on the other hand it pursues an approach rooted in observers, which is conceptually different than the approach pursued here; moreover, it is developed in continuous time and is focused on the challenge of solving such a problem without the measurements of the time derivatives of the outputs.
\cite{mao2025one} considers our same data-based output regulation problem in \cite[\S IV]{mao2025one} and solves it into two steps (corresponding to two linear matrix inequalities) under an assumption on a certain Krylov matrix \cite[Assumption~4]{mao2025one}, which is not needed in our one-step solution;
the effectiveness of the control design is proven in the case of state measurements and its extension is proposed in \cite[Algorithm 2]{mao2025one} without formal proofs; it is unclear whether this extension is viable in the case of output measurement from MIMO systems whereas our results are shown to hold for (a class of) MIMO systems, in particular by rigorously establishing a relation between the solutions of the auxiliary and the actual systems.

\subsection{Outline}

Section~\ref{sec:problem} formulates the output regulation problem to be solved from input-output data corrupted by noise in the form of an unknown exosignal and associates an auxiliary system with the actual linear system generating the data.
Section~\ref{sec:noise_indep_cl} obtains a closed-loop representation of the auxiliary system that is independent of the unknown exosignal.
Section~\ref{sec:output_reg_aux_act_sys} proposes our solution by first addressing a data-driven output regulation problem on the auxiliary system, in Section~\ref{sec:output_reg_aux_sys}, and then extending it to the considered problem on the actual system, in Section~\ref{sec:ddctrl:sys-auxsys}, by rigorously relating the input-output evolutions of auxiliary and actual system.
Section~\ref{sec:numerical_example} validates the effectiveness of the proposed method on a numerical example.
Proofs are in the appendix.

\subsection{Notation}
\label{sec:notation}
The set of integers is $\Z$.
The set of natural numbers, zero included, is $\Z_{\ge 0}$.
The identity matrix of size $n$ and the zero matrix of size $m$-by-$n$ are $I_n$ and $0_{m,n}$ wherein the indices are dropped when no confusion arises.
The set of $n \times n$ symmetric matrices is $\mathbb{S}^n$.
For a symmetric matrix $M$, $M \succ 0$ means that $M$ is positive definite.
For $A_1 \in \C^{n_1 \times n_1}$, \dots, $A_k \in \C^{n_k \times n_k}$, the block diagonal matrix $\smat{A_1 & \dots & 0\\ \vdotsS & \ddotsS & \vdotsS\\ 0 & \dots & A_k}$ is written $A_1 \oplus \dots \oplus A_k$, which is the direct sum of matrices $A_1$, \dots, $A_k$.
The imaginary unit is $\i$.
For $a \in \R$ and $b \in \R$, the complex conjugate of $a+ \i b$ is $\overline{a+ \i b} = a - \i b$.
For integers $p$ and $q$ such that $p \ge q \ge 0$, ${p \choose q} := \frac{p!}{q!(p-q)!}$;
for integers $p$ and $q$ such that $p \ge q$, $p \ge 0$ and $q < 0$, ${p \choose q} :=0$.
For column vectors $x_1$ and $x_2$, $(x_1,x_2) = \smat{x_1\\ x_2}$.

\section{Problem formulation}
\label{sec:problem}

Consider the discrete-time linear system
\begin{subequations}\label{eq:sys}
\begin{align}
w^+ &= S w\label{eq:sys:w}\\
x^+ &= A x + B u + P w\label{eq:sys:x}\\
y &= C x + Q w\label{eq:sys:y}
\end{align}
\end{subequations}
where $w \in \mathbb{R}^{n_w}$, $x \in \mathbb{R}^{n}$, $u \in \mathbb{R}^{m}$ and $y \in \mathbb{R}^{p}$ are exosignal, state, input and output, respectively.
The matrices $A$, $B$, $P$, $C$, $Q$ are unknown to us whereas $S \in \mathbb{R}^{n_w \times n_w}$ is known.
The exosignal $w$ and the state $x$ are not measured and only input $u$ and output $y$ are.
We rely on the measurements of $u$ and $y$, affected by the unknown $w$, to compensate for the lack of knowledge of $A$, $B$, $P$, $C$, $Q$ and to design a control law for $u$ that solves an output regulation problem, as specified below in Problem~\ref{prob:orp_ori}.

We make the next assumption on the matrix $S$, which is standard in the context of output regulation.
\begin{assumption}
\label{ass:S}
All eigenvalues of $S$ belong to $\{ z \in \C \colon |z| \ge 1\}$.
\end{assumption}

Since $S$ is known, we construct the internal model
\begin{equation}\label{eq:internal_model_eta}
    \eta^+ = \Phi \eta + G y
\end{equation}
of the exosystem where
\begin{align}
\label{eq:im-matrices}
\Phi :=
\smat{
0 & I_p & 0 & \dots & 0\\
0 & 0 & I_p & \dots & 0\\
\vdotsS & \vdotsS & \vdotsS & \ddotsS & \vdotsS\\
0 & 0 & 0 & \dots & I_p\\
-s_0 I_p &-s_1 I_p &-s_2 I_p &\ldots & -s_{d-1} I_p 
} \in \mathbb{R}^{pd \times pd},
\quad G := 
\smat{
0 \\ 0\\ \vdotsS \\ 0 \\I_p
} \in \mathbb{R}^{pd \times p}
\end{align}
and  $s_0, \dots, s_{d-1}$ are the coefficients of the minimal polynomial $m_{S}$ of $S$ \cite[Def.~3.3.2]{horn2013matrix}, i.e., $m_{S}(\lambda) := s_0 + s_1 \lambda + \dots + s_{d - 1} \lambda^{d - 1} + \lambda^d$.

If the pair $(A,C)$ in~\eqref{eq:sys} is observable, there exists $l$ such that
\begin{align*}
\rank \smat{C \\ C A  \\ \vdotsS \\ C A^{l-1}} = n
\end{align*}
and the minimum $l$ for which this rank condition holds is $\ell$, which is called the observability index (of pair $(A,C)$) \cite[p.~356-357]{kailath1980linear}.
In the considered input-output setting, our prior knowledge on~\eqref{eq:sys} is thus summarized next.

\begin{assumption}
\label{ass:observ}
The pair $(A, C)$ is observable and the observability index $\ell$ of $(A, C)$ is known.
\end{assumption}

Based on the observability index $\ell$ we define the matrices 
\begin{subequations}
\label{eq:pre:matrix:OTR}
\begin{align}
\Obs & := 
\smat{
C \\ C A  \\ \vdotsS \\ C A^{\ell-1}
} \label{eq:pre:matrix:OTR:O}\\
\Tu & :=
\smat{
0 & 0 & \dots & 0 & 0\\
CB                     & 0 & \dots & 0 & 0\\
CAB                    & CB                     & \dots & 0 & 0\\
\vdotsS                & \vdotsS                 & \ddotsS &  \vdotsS & \vdotsS\\
CA^{\ell -2} B         & CA^{\ell -3} B         & \dots & CB &0
} \label{eq:pre:matrix:OTR:T_B}\\
\Tw & :=
\smat{
Q & 0 & \dots & 0 & 0\\
CP                     & Q & \dots & 0 & 0\\
CAP                    & CP                     & \dots & 0 & 0\\
\vdotsS                & \vdotsS                 & \ddotsS &  \vdotsS & \vdotsS\\
CA^{\ell -2} P         & CA^{\ell -3} P         & \dots & CP &Q
} \label{eq:pre:matrix:OTR:T_P}\\
\Ru & := \smat{ A^{\ell-1} B & \dots & A B & B } \label{eq:pre:matrix:OTR:R_B}\\
\Rw & := \smat{ A^{\ell-1} P & \dots & A P & P } \label{eq:pre:matrix:OTR:R_P}
\end{align}
\end{subequations}
where $\Obs$ possesses a left inverse $\Obs^{\tu{L}}$, which satisfies $\Obs^{\tu{L}} \Obs = I_n$, since $\rank \Obs = n$.

Apart from Assumption~\ref{ass:observ}, our knowledge of~\eqref{eq:sys} is based on data collected in an experiment on~\eqref{eq:sys} and \eqref{eq:internal_model_eta}.
The experiment returns, for some integer $T \ge \ell$,
\begin{subequations}
\label{eq:data}
\begin{align}
U_1  &:= 
\smat{u(\ell) & u(\ell+1) & \dots & u(T)}\label{eq:data:V1}\\
\Psi_0  &:= 
\left[
\begin{smallarray}{cccc}
y(0) & y(1) & \dotsS & y(T-\ell)\\
\vdotsS & \vdotsS &   & \vdotsS \\
y(\ell-1) & y(\ell) & \dots & y(T-1) \\
\hline
u(0) & u(1) & \dotsS & u(T-\ell) \rule{0pt}{8pt}\\
\vdotsS & \vdotsS &  & \vdotsS\\
u(\ell-1) & u(\ell) & \dots & u(T-1)\\
\hline
    \eta(\ell) & \eta(\ell + 1) & \cdots & \eta(T)
\end{smallarray}
\right] \label{eq:data:Psi0}
\\
\Psi_1  &:=
\left[
\begin{smallarray}{cccc}
y(1) & y(2) & \dots & y(T-\ell+1)\\
\vdotsS & \vdotsS &   & \vdotsS \\
y(\ell) & y(\ell+1) & \dots & y(T)\\ 
\hline 
u(1) & u(2) & \dots & u(T-\ell+1) \rule{0pt}{8pt}\\
\vdotsS & \vdotsS &  & \vdotsS\\
u(\ell) & u(\ell+1) & \dots & u(T)\\ 
\hline
\eta(\ell + 1) & \eta(\ell + 2) & \cdots & \eta(T+1)
\end{smallarray}
\right]; \label{eq:data:Psi1}
\end{align}
\end{subequations}
in the experiment, the action of the unknown $w$ leads to the unknown
\begin{align}
W_0 &:= 
\smat{w(\ell) & w(\ell + 1) & \cdots & w(T)
} \label{eq:W0}.
\end{align}

With all these elements, we can formulate our objective, i.e., to solve an output regulation problem \cite[p.~84]{isidori2017lectures} \emph{from data}.

\begin{problem}
\label{prob:orp_ori}
With data $U_1$, $\Psi_0$, $\Psi_1$, find a control law for $u$ in~\eqref{eq:sys} so that
    \begin{itemize}[wide]
        \item [(i)] the resulting closed loop with $w = 0$ is globally asymptotically stable;
        \item [(ii)] the resulting closed loop has $\lim_{k \to + \infty} y(k)= 0$ for each initial state.
    \end{itemize}
\end{problem}

Thanks to Assumptions~\ref{ass:S}-\ref{ass:observ}, we can associate, with~\eqref{eq:sys}, an auxiliary system%
\begin{subequations}\label{eq:sys:aux}
\begin{align}
\omega^+ &= S \omega	\label{eq:sys:aux:w}\\
\xi^+ & = \Ab \xi + \Bb v  + \Lb Z_{2} \Sb \omega		\label{eq:sys:aux:xi}    \\
\varphi & = Z_{1} \xi + Z_{2} \Sb \omega		\label{eq:sys:aux:y}
\end{align}
\end{subequations}
where
\begin{subequations}\label{eq:pre:matrix_SABZ}
\begin{align}
    &\begin{bmatrix}
\Fb & \Lb & \Bb
\end{bmatrix} \\
& \quad := 
\begin{bmatrix}
\left[ \begin{array}{c|c}
\smatNoB{
0 & I_p & 0 & \cdots & 0 \\ 
0 & 0 & I_p & \cdots & 0 \\ 
\vdotsS & \vdotsS & \vdotsS & \ddotsS & \vdotsS \\ 
0 & 0 & 0 & \cdots & I_p \\ 
0 & 0 & 0 & \cdots & 0
} 
& 0_{p \ell, m \ell} \\ 
\hline 
0_{m \ell, p \ell} & 
\smatNoB{
0 & I_m & 0 & \cdots & 0 \\ 
0 & 0 & I_m & \cdots & 0 \\ 
\vdotsS & \vdotsS & \vdotsS & \ddotsS & \vdotsS \\ 
0 & 0 & 0 & \dots & I_m \\ 
0 & 0 & 0 & \dots & 0
} 
\end{array} \right]
&
\left[%
\begin{array}{c}
\smatNoB{
0 \\[2pt] 
0 \\[1pt] 
\vdotsS \\[1pt] 
0 \\[1pt]
I_p } \\[1pt] 
\hline  
\smatNoB{
0 \\[1pt]  
0 \\[1pt] 
\vdotsS \\[1pt] 
0 \\[1pt] 
0} 
\end{array} \right]
&
\left[ \begin{array}{c}
\smatNoB{ 
0 \\[2pt] 
0 \\[1pt] 
\vdotsS \\[1pt] 
0 \\[1pt]
0\rotatebox{180}{\rule{0pt}{2.5pt}}
} \\ 
\hline  
\smatNoB{
0 \\[1pt] 
0 \\[1pt] 
\vdotsS \\[1pt] 
0 \\[1pt] 
I_m
} 
\end{array} \right]
\end{bmatrix},\\
& \Ab := \Fb + \Lb Z_1, \quad \Sb := \smat{S^{-\ell} \\ \vdotsS \\ S^{-1} \\ I_{n_w} },\label{eq:pre:matrix_SABZ:S} \\
& Z_1 := 
\begin{bmatrix}
  CA^{\ell} \Obs^{\tu{L}} &   C\Ru - CA^{\ell} \Obs^{\tu{L}}\Tu
\end{bmatrix},~~Z_2 := 
\begin{bmatrix}
   C\Rw - CA^{\ell} \Obs^{\tu{L}}\Tw & Q
\end{bmatrix}\label{eq:pre:matrix_SABZ:Z}.
\end{align}    
\end{subequations}
$\Sb$ is well-defined since $S$ is invertible by Assumption~\ref{ass:S}.
We consider this auxiliary system motivated by the approach in~\cite{de2019formulas,li2024controller}, which we summarize here.
Indeed, the auxiliary system can capture the input-output evolution of the actual system: for each $(\hat{w}, \hat{x})$ and sequence $\{u(k)\}_{k = 0}^{\infty}$, there exists $\hat{\xi}$ such that 
\begin{itemize}[leftmargin=*]
    \item the solution $\big(w(\cdot), x(\cdot)\big)$ to \eqref{eq:sys:w}-\eqref{eq:sys:x} with initial condition $\big(w(0), x(0) \big)= \big(\hat{w}, \hat{x} \big)$ and with input $\{u(k)\}_{k = 0}^{\infty}$,
    \item the corresponding output response $y(\cdot)$ in~\eqref{eq:sys:y},
    \item the solution $\big(\omega(\cdot),\xi(\cdot)\big)$ to \eqref{eq:sys:aux:w}-\eqref{eq:sys:aux:xi} with initial condition $(\omega(\ell), \xi(\ell)) = \big( S^\ell \hat{w}, \hat{\xi} \big)$ 
    and input $\{v(k)\}_{k = \ell}^{\infty} = \{u(k)\}_{k = \ell}^{\infty}$
    \item the corresponding output response $\varphi(\cdot)$ in~\eqref{eq:sys:aux:y}
\end{itemize}
satisfy for each $k \ge \ell$
\begin{align*}
\xi(k) = \big(y(k - \ell), \dots, y(k - 1), u(k - \ell), \dots, u(k - 1) \big) \text{ and }  \varphi(k) = y(k).
\end{align*}
(We will make this claim in Lemma~\ref{lem:rel_sol_sys+exo_auxsys+auxexo}.)
The fact that the $\xi$ above is a stack of $\ell$ past values of the output and the input is why we arranged data as in~\eqref{eq:data} and created stacks of $\ell$ past values of output and input in~\eqref{eq:data:Psi0} and \eqref{eq:data:Psi1}.
Since the auxiliary system can capture the input-output evolution of the actual system in the sense above, we first consider an output regulation problem for~\eqref{eq:sys:aux}; then, based on the solution to that problem and the relation between the evolutions of the two systems, we construct a control law $u$ that solves Problem~\ref{prob:orp_ori}.
The considered output regulation problem for~\eqref{eq:sys:aux} is presented next.

\begin{problem}
\label{prob:orp_aux}
    Find a control law for $v$ in~\eqref{eq:sys:aux} so that
    \begin{itemize}[wide]
        \item [(i)] the resulting closed loop with $\omega = 0$ is globally asymptotically stable;
        \item [(ii)] the resulting closed loop has $\lim_{k \to + \infty} \varphi(k)= 0$ for each initial state.
    \end{itemize}
\end{problem}

To solve Problem \ref{prob:orp_aux}, we consider
\begin{equation}
\label{eq:ctrl:aux}
\mu^+ = \Phi\mu + G \varphi, \quad v = \mathcal{K}\smat{\xi\\ \mu}
\end{equation}
for $\Phi$ and $G$ in~\eqref{eq:im-matrices} and combine it with~\eqref{eq:sys:aux} to obtain the closed-loop
\begingroup
\thinmuskip=1mu plus 1mu minus 1mu
\medmuskip=2mu plus 2mu minus 2mu
\thickmuskip=3mu plus 3mu minus 3mu
\begin{subequations}\label{eq:sys:xi_eta}
    \begin{align}
    \omega^+ & = S \omega \\
    \bmat{\xi^+\\\mu^+}
	& = 
	\bmat{\Ab & 0\\  G Z_{1} & \Phi}
    \bmat{\xi\\ \mu}
    + \bmat{\Bb\\ 0} v 
    + \bmat{\Lb Z_{2} \\ GZ_{2}} \Sb \omega =: \bar{\Ab} \bmat{\xi\\ \mu}
	+ \bar{\Bb} v + \bar{\Pb} \Sb \omega \label{eq:sys:xi_eta:barAbarB} \\
        \varphi & = Z_{1} \xi + Z_{2} \Sb \omega, \quad v = \mathcal{K}\smat{\xi\\ \mu}.\label{eq:sys:xi_eta:phi}
    \end{align}
\end{subequations}
\endgroup

\section{Noise-independent data-based closed-loop representation}
\label{sec:noise_indep_cl}

In this section, we show how the closed loop can be expressed in terms of the available data, which do not include the matrix $W_0$ corresponding to the unmeasurable exosignal.
To this end, it is relevant to characterize the solutions to~\eqref{eq:sys} in the next claim, by adapting \cite[Claim 1]{li2024controller} in the additional presence of the exosystem $w^+ = S w$.

\begin{claim}
\label{claim:1:xy}
    For each $(\hat{w},\hat{x}) \in \mathbb{R}^{n_w} \times \mathbb{R}^n$ and sequence $\{u(k)\}_{k = 0}^{\infty}$, let $(w(\cdot), x(\cdot))$ be the solution to~\eqref{eq:sys:w}-\eqref{eq:sys:x} with initial condition $(w(0), x(0)) = (\hat{w}, \hat{x})$ and input $\{u(k)\}_{k = 0}^{\infty}$, and let $y(\cdot) = C x(\cdot) + Q w(\cdot)$ be the corresponding output response in~\eqref{eq:sys:y}.
    If $(A,C)$ is observable, then, for each $k \ge \ell$,
    \begin{subequations}\label{eq:claim1:yellxy}
        \begin{align}
            &
            \smat{
                y(k - \ell)\\
                \vdotsS\\
                y(k - 1)} = \Obs x(k - \ell) + \Tu 
            \smat{
                u(k - \ell)\\
                \vdotsS\\
                u(k - 1)} + \Tw 
			\smat{
                w(k - \ell)\\
                \vdotsS\\
                w(k - 1)} \label{eq:claim1:yellxy:yell}\\
            & x(k) = A^\ell \Obs^{\tu{L}}
            \smat{
                y(k - \ell)\\
                \vdotsS\\
                y(k - 1)}
                + (\Ru - A^\ell \Obs^{\tu{L}} \Tu)
            \smat{
                u(k - \ell)\\
                \vdotsS\\
                u(k - 1)} 
            \label{eq:claim1:yellxy:x}\\
			& \qquad + (\Rw - A^\ell \Obs^{\tu{L}} \Tw)
            \smat{
                w(k - \ell)\\
                \vdotsS\\
                w(k - 1)}  \notag \\           
            & y(k) = CA^{\ell} \Obs^{\tu{L}}
            \smat{
			y(k-\ell)\\
			\vdotsS\\
			y(k-1)            
            }  + (C\Ru - CA^{\ell} \Obs^{\tu{L}}\Tu)
			\smat{
			u(k-\ell)\\
			\vdotsS \\
			u(k-1)
			} \label{eq:claim1:yellxy:y} \\
			& \qquad + (C\Rw - CA^{\ell} \Obs^{\tu{L}}\Tw)
			\smat{
			w(k-\ell)\\
			\vdotsS \\
			w(k-1)
			}
			+ Q w(k). \notag
        \end{align}
    \end{subequations}
\end{claim}
\begin{proof}
See Appendix~\ref{app:claim:1:xy}.
\end{proof}

By Claim~\ref{claim:1:xy}, a key relation for data points is obtained in the next lemma, which is an adaptation of \cite[Lemma 4]{li2024controller}.

\begin{lemma}
\label{lemma:data_rel}
Under Assumptions~\ref{ass:S}-\ref{ass:observ}, data points in~\eqref{eq:data}-\eqref{eq:W0} satisfy
\begin{align}
\label{eq:data-matrices-identity}
\Psi_1 = \bar{\Ab}  \Psi_0 +  \bar{\Bb} U_1 + \bar{\Pb} \Sb W_0.
\end{align}
\end{lemma}
\begin{proof}
See Appendix~\ref{app:lemma:data_rel}.
\end{proof}

The data matrices $\Psi_1$, $\Psi_0$, $U_1$ are available to the designer, but the matrix $W_0$ is not.
A factorization of $W_0$ is instrumental to obtain a closed-loop representation independent of $W_0$ and is given in the next lemma.

\begin{lemma}
\label{lemma:factorization_W0}
There is a matrix $L_0$ and a \emph{known} matrix $\mathbf{M}\in \R^{n_w \times T-\ell+1}$ such that
\begin{align}
\label{eq:W_0:factor}
W_0 = L_0 \mathbf{M}.
\end{align}
\end{lemma}
\begin{proof}
See Appendix~\ref{app:lemma:factorization_W0}.
\end{proof}
In the proof in Appendix~\ref{app:lemma:factorization_W0}, the construction of $\mathbf{M}$ is detailed.

A factorization of $W_0$ as in~\eqref{eq:W_0:factor} can be alternatively obtained under an additional assumption on $S$ and $T$ as in \cite{mao2025one}; we report it in the next lemma  for self-containedness.

\begin{lemma}
\label{lemma:factorization_W0_new}
Find $w_\star$ such that the pair $(S,w_\star)$ is controllable, i.e.,
\begin{align*}
\bmat{ w_\star & S w_\star & \dots & S^{T-\ell} w_\star} =: \newM \in \R^{n_w \times T-\ell+1}
\end{align*}
has full row rank.
Then, there is a matrix $\mathcal{L}$ such that
\begin{align}
\label{factorization_W0_new}
W_0 = \mathcal{L} \newM.
\end{align}
\end{lemma}
\begin{proof}
See Appendix~\ref{app:lemma:factorization_W0_new}.
\end{proof}

The hypothesis that $\newM$ has full row rank implies that $T-\ell+1 \ge n_w$, which is not required by Lemma~\ref{lemma:factorization_W0}.

Whenever a factorization of $W_0$ as in Lemma~\ref{lemma:factorization_W0} or \ref{lemma:factorization_W0_new} is available, the next result applies.

\begin{lemma}
\label{lemma:closed_loop_rep}
Let a matrix $\hat{L}_0$ and a \emph{known} matrix $\hat{\mathbf{M}} \in \R^{\hat{n}_w \times T-\ell+1}$ be such that $W_0 = \hat{L}_0 \hat{\mathbf{M}}$.
If there are matrices $\mathcal{K} \in \mathbb{R}^{m \times (m+p)\ell + pd}$ and $\mathcal{G} \in \mathbb{R}^{T - \ell + 1 \times (m+p)\ell + p d}$ such that
\begin{equation}
\label{eq:mathcal_G}
\bmat{\mathcal{K}\\
        I \\
        0_{\hat{n}_w,(m+p)\ell + pd}} 
= \bmat{U_1\\
        \Psi_0\\
        \hat{\mathbf{M}}} \mathcal{G},
\end{equation}
then $\bar{\Ab} +  \bar{\Bb}\mathcal{K} = \Psi_1 \mathcal{G}$.
\end{lemma}
\begin{proof}
See Appendix~\ref{app:lemma:closed_loop_rep}.
\end{proof}

As a consequence of Lemma~\ref{lemma:closed_loop_rep}, \eqref{eq:sys:xi_eta:barAbarB} and $v = \mathcal{K}\smat{\xi\\ \mu}$ yield the closed-loop dynamics
\begin{align*}
\smat{\xi^+\\ \mu^+} 
= (\bar{\Ab} +  \bar{\Bb}\mathcal{K}) \smat{\xi\\ \mu} + \bar{\Pb} \Sb \omega 
= \Psi_1 \mathcal{G}  \smat{\xi\\ \mu} + \bar{\Pb} \Sb \omega .
\end{align*}
The factorization $W_0 = \hat{L}_0 \hat{\mathbf{M}}$ applies to $W_0 = L_0 \mathbf{M}$ as per Lemma~\ref{lemma:factorization_W0} and to $W_0 = \mathcal{L} \mathcal{M}$ as per Lemma~\ref{lemma:factorization_W0_new}.
Then, because of the equation $0 = \hat{\mathbf{M}} \mathcal{G}$ in~\eqref{eq:mathcal_G}, Lemma~\ref{lemma:closed_loop_rep} is a noise-filtering lemma as it enables writing the closed-loop as $\bar{\Ab} +  \bar{\Bb}\mathcal{K} = \Psi_1 \mathcal{G}$, which is independent of the noise $W_0$ affecting data.

\section{Data-driven output regulation of auxiliary and actual systems}
\label{sec:output_reg_aux_act_sys}

Within this section, we solve the considered output regulation problem via input-output data.
In Section~\ref{sec:output_reg_aux_sys}, we address Problem~\ref{prob:orp_aux} given on the auxiliary system and in Section~\ref{sec:ddctrl:sys-auxsys} we address Problem~\ref{prob:orp_ori} on the actual system.

\subsection{Data-driven output regulation of the auxiliary system}
\label{sec:output_reg_aux_sys}

In the next theorem we propose a semidefinite program (SDP) to solve the output regulation problem for the auxiliary system \eqref{eq:sys:aux} using the input-output data in~\eqref{eq:data} collected from \eqref{eq:sys}-\eqref{eq:internal_model_eta}, as motivated in Section~\ref{sec:problem}.

\begin{theorem}
\label{thm:outputregulation:aux}
Let Assumptions \ref{ass:S}-\ref{ass:observ} hold, and a matrix $\hat{L}_0$ and a \emph{known} matrix $\hat{\mathbf{M}} \in \R^{\hat{n}_w \times T-\ell+1}$ be such that $W_0 = \hat{L}_0 \hat{\mathbf{M}}$.
Consider the SDP in the decision variables $\mathcal{X} \in \mathbb{S}^{(m+p)\ell + pd }$, $\mathcal{Y} \in \R^{T - \ell +1 \times (m+p)\ell + pd}$%
\begin{subequations}\label{eq:sdp}
\begin{align}
&\mathcal{X} \succ 0 , \label{eq:sdp:x>0}\\
& \bmat{
			\mathcal{X}\\
			0_{\hat{n}_w, (m+p)\ell + pd}
}=
\bmat{
\Psi_0\\ 
\hat{\mathbf{M}}
}\mathcal{Y} ,
\label{eq:sdp:mathcal_Y}\\
& \bmat{
\mathcal{X} &  \Psi_1 \mathcal{Y}\\
(\Psi_1 \mathcal{Y})^\top & \mathcal{X} 
} \succ 0. \label{eq:sdp:schur}
\end{align}
\end{subequations}
If \eqref{eq:sdp} is feasible, the control law \eqref{eq:ctrl:aux} with $\Phi$ and $G$ as in \eqref{eq:im-matrices} and
\begin{equation} \label{eq:K}
    \mathcal{K} := U_1 \mathcal{Y} \mathcal{X}^{-1}
\end{equation}
solves the output regulation problem for~\eqref{eq:sys:aux}, that is: the closed loop of \eqref{eq:sys:aux} and \eqref{eq:ctrl:aux} given in~\eqref{eq:sys:xi_eta} (i)~is globally asymptotically stable with $\omega=0$ and (ii)~has $\lim_{k\to\infty} \varphi(k) = 0$ for each initial condition $(\omega(0),\xi(0),\mu(0))$.
\end{theorem}
\begin{proof}
See Appendix~\ref{app:thm:outputregulation:aux}.
\end{proof}

This theorem corresponds to the solution to Problem~\ref{prob:orp_aux}.
It is worth remarking that due to the noise filtering property provided in Lemma \ref{lemma:closed_loop_rep}, no assumptions on the magnitude of the exosignal are imposed in contrast to the assumption of an energy bound on the disturbance as in~\cite[Eq.~(18)]{li2024controller}.
Finally, both factorizations of Lemma~\ref{lemma:factorization_W0} and Lemma~\ref{lemma:factorization_W0_new} can be used with this theorem, specifically in~\eqref{eq:sdp}.

\begin{remark}
A converse of the implication in Theorem~\ref{thm:outputregulation:aux} holds under additional assumptions, besides Assumptions~\ref{ass:S}-\ref{ass:observ} and the existence of $\hat{L}_0$ and $\hat{\mathbf{M}}$ such that $W_0 = \hat{L}_0 \hat{\mathbf{M}}$.
Namely, assume also that (i)~$\smat{ U_1 \\ \Psi_0\\ \hat{\mathbf{M}} }$ has full row rank and (ii)~$\smat{\Ab - \lambda I & \Bb\\ Z_1 & 0}$ has full row rank for all $\lambda$ in the spectrum of $S$; then, if the output regulation problem for system \eqref{eq:sys:aux} is solvable, there exist $\mathcal{X}$ and $\mathcal{Y}$ satisfying \eqref{eq:sdp}.
Condition (ii) is known as nonresonance condition in the output regulation literature and involves the unknown $Z_1$, see~\eqref{eq:pre:matrix_SABZ}.
Let us sketch how this converse implication can be shown.
If the output regulation problem for~\eqref{eq:sys:aux} is solvable, then the triplet $(\Ab, \Bb,\smat{Z_{1}\\ I})$ is stabilizable and detectable, by adapting the arguments of \cite[Proposition 4.2]{isidori2017lectures} given for continuous time and exosystem with eigenvalues on the imaginary axis.
By stabilizability of $\left(\Ab, \Bb\right)$ and condition~(ii), the pair $(\bar{\Ab}, \bar{\Bb}) = \left( \smat{\Ab & 0\\ \mathbf{G} Z_1 & \Phi}, \smat{\Bb\\ 0}\right)$ is stabilizable by following the same steps as \cite[Lemma 4.2 and its proof]{isidori2017lectures}.
By stabilizability of $(\bar{\Ab},\bar{\Bb})$, there exists $\mathcal{K}$ such that $\bar{\Ab} + \bar{\Bb} \mathcal{K}$ is Schur.
For this $\mathcal{K}$, there exists $\mathcal{G}$ such that%
\begin{align}
\label{conseq_full_row_rank}
\smat{\mathcal{K}\\
I\\
0_{\hat{n}_w, (m+p)\ell + pd}}
= \smat{         
U_1\\
\Psi_0\\
\hat{\mathbf{M}}
}\mathcal{G}
\end{align}
by condition~(i); this and the existence of $\hat{L}_0$ and $\hat{\mathbf{M}}$ such that $W_0 = \hat{L}_0 \hat{\mathbf{M}}$ ensure $\bar{\Ab} + \bar{\Bb} \mathcal{K} = \Psi_1 \mathcal{G}$ by Lemma~\ref{lemma:closed_loop_rep} and, thus, $\Psi_1 \mathcal{G}$ Schur.
Then, there exist $\mathcal{X}$ and $\mathcal{Y} = \mathcal{G}\mathcal{X}$ such that \eqref{eq:sdp:x>0} and \eqref{eq:sdp:schur} hold; \eqref{eq:sdp:mathcal_Y} holds by the last two block rows of \eqref{conseq_full_row_rank}.
\end{remark}

A relevant implication regarding~\eqref{eq:sdp} is contained in the next result.

\begin{proposition}
\label{proposition:ex_remark}
Suppose Assumptions~\ref{ass:S}-\ref{ass:observ} hold, and suppose there is a matrix $\hat{L}_0$ and a \emph{known} matrix $\hat{\mathbf{M}} \in \R^{\hat{n}_w \times T-\ell+1}$ such that $W_0 = \hat{L}_0 \hat{\mathbf{M}}$. If $\hat{\mathbf{M}}$ is full row rank and $p \ell > n$, then \eqref{eq:sdp} is infeasible.
\end{proposition}
\begin{proof}
See Appendix~\ref{app:proposition:ex_remark}.
\end{proof}

Assumption~\ref{ass:observ} implies $p \ell \ge n$.
Proposition~\ref{proposition:ex_remark} establishes that whenever $W_0$ can be factorized as $W_0 = \hat{L}_0 \hat{\mathbf{M}}$ with $\hat{\mathbf{M}}$ full row rank, then $p \ell = n$ is a \emph{necessary} condition for~\eqref{eq:sdp} to be feasible.
This applies to~\cite{mao2025one}, where $W_0 = \mathcal{L} \mathcal{M}$ with $\mathcal{M}$ full row rank, see Lemma~\ref{lemma:factorization_W0_new}.
In practice, this also applies to the proposed factorization $W_0 = L_0 \mathbf{M}$ in Lemma~\ref{lemma:factorization_W0}.
Indeed, although $\mathbf{M}$ does not have full row rank in general, one can always obtain a matrix $\hat{\mathbf{M}}$ with full row rank from $\mathbf{M}$ by removing all linearly dependent rows from $\mathbf{M}$ and write $\mathbf{M} = \mathcal{S}_{\mathbf{M}} \hat{\mathbf{M}}$ for some matrix $\mathcal{S}_{\mathbf{M}}$ so that $W_0 = L_0 \mathbf{M} = (L_0 \mathcal{S}_{\mathbf{M}} ) \cdot \hat{\mathbf{M}}$.
The relevance of the condition $p \ell = n$ was already pointed out in~\cite{li2024controller} for the stabilization problem with input-output data. Removing this condition remains to be investigated in view of the results of \cite{alsalti2025notes,li2024controller,lee2025input}.

\subsection{From output regulation of the auxiliary system to that of the actual system}
\label{sec:ddctrl:sys-auxsys}

\begin{figure}
	\centerline{\includegraphics[scale=.7]{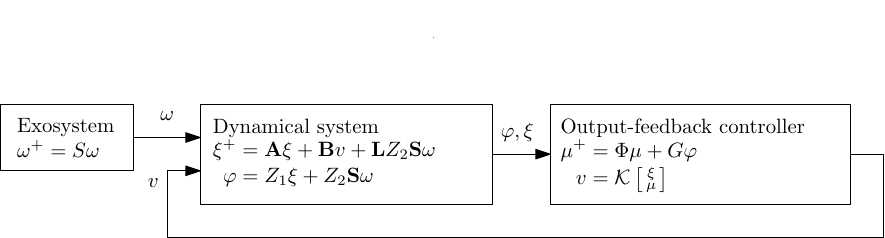}}
	\caption{Plant in \eqref{eq:sys:aux} with dynamical controller in  \eqref{eq:ctrl:aux}.}
	\label{fig:aug_sys}
\end{figure}

Theorem \ref{thm:outputregulation:aux} solves the output regulation problem for~\eqref{eq:sys:aux}.
In order that the controller gain $\mathcal{K}$ tuned by Theorem \ref{thm:outputregulation:aux} solves the output regulation problem for~\eqref{eq:sys}, we need to investigate the relation between the solutions to~\eqref{eq:sys:aux} and to~\eqref{eq:sys}, once the controller gain $\mathcal{K}$ is used in feedback.
More precisely, the feedback controller used in combination with \eqref{eq:sys:aux} is \eqref{eq:ctrl:aux}, see Figure~\ref{fig:aug_sys}, and that used in combination with~\eqref{eq:sys} is
\begin{subequations}\label{eq:pre:ctrl_sys}
\begin{align}
          \chi^+ &= \Fb \chi + \Lb y + \Bb u\label{eq:pre:ctrl_sys:chi}\\
          \eta^+ &= \Phi \eta + G y\label{eq:pre:ctrl_sys:eta}\\
         u &= \mathcal{K} \smat{\chi \\ \eta},
     \end{align}
\end{subequations}
see Figure~\ref{fig:plant_dyn_ctrl}.
\begin{figure}
	\centerline{\includegraphics[scale=.7]{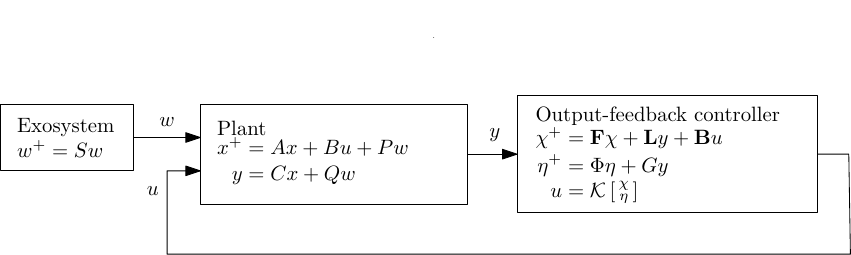}}
	\caption{Plant and output-feedback controller.}
	\label{fig:plant_dyn_ctrl}
\end{figure}
Let us equivalently rearrange the scheme in Figure~\ref{fig:plant_dyn_ctrl} as
in Figure~\ref{fig:ori_sys}.
To establish the desired relation between the solutions of the closed-loop \eqref{eq:sys:aux}, \eqref{eq:ctrl:aux} in Figure~\ref{fig:aug_sys} and those of the closed-loop \eqref{eq:sys}, \eqref{eq:pre:ctrl_sys} in Figure~\ref{fig:ori_sys}, we first single out the equations
\begin{subequations}\label{eq:aux_all_except_control}
\begin{align}
\omega^+ &= S \omega \label{eq:aux_all_except_control:omega}\\
\xi^+ & = \Ab \xi + \Bb v  + \Lb Z_{2} \Sb \omega \label{eq:aux_all_except_control:xi}\\
\varphi &= Z_1 \xi + Z_2 \Sb \omega, \label{eq:aux_all_except_control:varphi}
\end{align}
\end{subequations}
which appear to the left of Figure~\ref{fig:aug_sys}, and the equations
\begin{subequations}\label{eq:pre:sys_cl_xy}
    \begin{align}
w^+ &= S w \label{eq:pre:sys_cl_xy:w}\\        
x^+ & = A x + B u + Pw \label{eq:pre:sys_cl_xy:x}\\
\chi^+ &= \Fb \chi + \Lb y + \Bb u\label{eq:pre:sys_cl_xy:chi}\\
y &= C x + Qw,  \label{eq:pre:sys_cl_xy:y}
     \end{align}
\end{subequations}
which appear to the left of Figure~\ref{fig:ori_sys}.

\begin{figure}
	\centerline{\includegraphics[scale=.7]{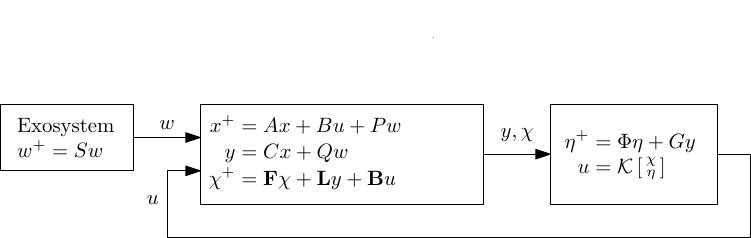}}
	\caption{Equivalent reformulation of the scheme in Figure~\ref{fig:plant_dyn_ctrl}.}
	\label{fig:ori_sys}
\end{figure}

We extend \cite[Lemmas 3 and 7]{li2024controller} in the additional presence of the exosystems $w^+ = S w$ and $\omega^+ = S \omega$ to establish a relation between the input-output evolutions of~\eqref{eq:aux_all_except_control} and \eqref{eq:pre:sys_cl_xy}.

\begin{lemma}
\label{lem:rel_sol_sys+exo_auxsys+auxexo}
    Let $(A,C)$ be observable.
    For each $(\hat{w}, \hat{x}, \hat{\chi})$ and sequence $\{u(k)\}_{k = 0}^{\infty}$, there exists $\hat{\xi}$ such that:
\begin{itemize}[leftmargin=*]
    \item the solution $\big(w(\cdot), x(\cdot), \chi(\cdot)\big)$ to \eqref{eq:pre:sys_cl_xy} with initial condition $\big(w(0), x(0), \chi(0)\big)= \big(\hat{w}, \hat{x}, \hat{\chi}\big)$ and with input $\{u(k)\}_{k = 0}^{\infty}$,
    \item the corresponding output response $y(\cdot) = C x(\cdot) + Q w(\cdot)$;  and
    \item the solution $\big(\omega(\cdot),\xi(\cdot)\big)$ to \eqref{eq:aux_all_except_control} with initial condition equal to $(\omega(\ell), \xi(\ell)) = \big( S^\ell \hat{w}, \hat{\xi} \big)$ and input $\{v(k)\}_{k = \ell}^{\infty} = \{u(k)\}_{k = \ell}^{\infty}$,
    \item the corresponding output response $\varphi(\cdot) = Z_1 \xi(\cdot) + Z_2 \Sb \omega(\cdot)$
\end{itemize}
satisfy for each $k \ge \ell$
\begin{align*}
& \xi(k) = \left[\begin{smallarray}{c}
        y(k - \ell)\\
        \vdotsS\\
        y(k - 1) \\
        \hline
        u(k - \ell)\\
        \vdotsS\\
        u(k - 1)
        \end{smallarray}
        \right] = \chi(k),
\quad \varphi(k) = y(k).
\end{align*}
\end{lemma}
\begin{proof}
See Appendix~\ref{app:lem:rel_sol_sys+exo_auxsys+auxexo}.
\end{proof}

Building on this lemma, the next theorem extends the stabilization setting of \cite[Lemma 8 and Theorem 1]{li2024controller} to solve the output regulation problem.

\begin{theorem}\label{thm:outreg_sys=outreg_auxsys}
Let the assumptions of Theorem~\ref{thm:outputregulation:aux} hold and \eqref{eq:sdp} be feasible so that $\mathcal{K}$ is as in~\eqref{eq:K}.
Then, controller \eqref{eq:pre:ctrl_sys} solves the output regulation problem for~\eqref{eq:sys}, that is: the closed-loop of~\eqref{eq:sys} and \eqref{eq:pre:ctrl_sys} given by
\begin{subequations}\label{eq:sys_cl_xy2}
    \begin{align}
    w^+ &= S w\label{eq:sys_cl_xy2:w}\\
         x^+ & = A x + B u + Pw\label{eq:sys_cl_xy2:x}\\
         y &= C x + Qw\label{eq:sys_cl_xy2:y}\\
         \chi^+ &= \Fb \chi + \Lb y + \Bb u\label{eq:sys_cl_xy2:chi}\\
         \eta^+ &= \Phi \eta + G y\label{eq:sys_cl_xy2:eta}\\
         u & = \mathcal{K} \left[
         \begin{smallmatrix}
             \chi\\
             \eta
         \end{smallmatrix}
         \right] \label{eq:sys_cl_xy2:u}
     \end{align}
\end{subequations}
(i)~is globally asymptotically stable with $w=0$ and (ii)~has $\lim_{k\to\infty} y(k) = 0$ for each initial condition $(w(0),x(0),\chi(0),\eta(0))$.
\end{theorem}
\begin{proof}
See Appendix~\ref{app:thm:outreg_sys=outreg_auxsys}.
\end{proof}

This theorem corresponds to the solution to Problem~\ref{prob:orp_ori} and shows how to achieve output regulation from input-output data.

\section{Numerical example}
\label{sec:numerical_example}

In the system to be controlled in~\eqref{eq:sys}, consider
\begin{align*}
& S = 
\smat{
     0 &    1\\
    -1 &    0}
\\
& A = \smat{
    1.0000 &   0.2500 &        0 &        0\\
         0 &   1.0000 &        0 &        0\\
   -0.3066 &  -0.0255 &   1.0000 &   0.2500\\
   -2.4525 &  -0.3066 &        0 &   1.0000
   },  
B = \smat{
    0.4396\\
    3.5172\\
   -0.0152\\
   -0.3015
   }, 
P = \smat{
     0 &    0\\
     0 &    0\\
     0 &    0\\
     0 &    0} \\
& C=\smat{
     1 &    0 &    1 &   0},  
Q =\smat{
     1 &    0},
\end{align*}
which is a discretization of the continuous-time linearized model in \cite[p.~67]{isidori2017lectures} for a vertical-take-off-and-landing aircraft.
It can be checked that the matrix $A$ is similar to $J_4(1)$, see \eqref{jordan_block}, so that the system is unstable.
As discussed in Section~\ref{sec:problem}, the matrices $A$, $B$, $P$, $C$, $Q$ are unknown; we know the observability index $\ell = 4$, see Assumption~\ref{ass:observ}; the matrix $S$ is known and its characteristic and minimal polynomial is $m_S(\lambda) = 1 + \lambda^2$ with eigenvalues $\pm \i$ so that the exosignal is sinusoidal with known frequency, but unknown amplitude and phase.
From $S$ we obtain $\Phi = \smat{0 & 1\\ -1 & 0}$ and $G = \smat{0\\ 1}$ as in~\eqref{eq:im-matrices}.
Data are collected for unknown $w(0)=(0.0538,0.1834)$ and $x(0)=(-2.2588,0.8622,0.3188,-1.3077)$; we set $\eta(0) = (-0.4336,0.3426)$ and the input $u$ to contain pseudorandom values drawn from a standard normal distribution.
The collected data are $\{ u(t_i), y(t_i) \}_{i=0}^{T}$
with $T = 20$ and, based on them, the matrices $U_1$, $\Psi_0$, $\Psi_1$ are constructed as in~\eqref{eq:data}.
From $S$, the Jordan block associated with the eigenvalue $\i = \e^{\i \pi/2}$ has order 1 so, from Appendix~\ref{app:lemma:factorization_W0} and specifically \eqref{boldM}, we construct
\begin{align*}
\mathbf{M} =
\smat{
\cos(\frac{\pi}{2}\ell) & \cos(\frac{\pi}{2}(\ell+1)) & \dots & \cos(\frac{\pi}{2}T)\\
\sin(\frac{\pi}{2}\ell) & \sin(\frac{\pi}{2}(\ell+1)) & \dots & \sin(\frac{\pi}{2}T)
}.
\end{align*}
As the key step in the proposed approach, the matrices $\mathbf{M}$, $U_1$, $\Psi_0$, $\Psi_1$ are used to design, via the SDP in Theorem~\ref{thm:outputregulation:aux}, the controller
\begin{align*}
\mathcal{K} = 
\smat{
-12.3406 &  41.7100 & -47.7489 &  18.6308 &   5.2395 &  -5.3061 & -12.0506  & -2.2147 &  -0.0905 &   0.0881
}.
\end{align*}
This controller is used in the closed loop~\eqref{eq:sys_cl_xy2} and the validity of Theorem~\ref{thm:outreg_sys=outreg_auxsys} is confirmed by the solutions in Figure~\ref{fig:closed_loop}, where the satisfaction of items (i) and (ii) in Theorem~\ref{thm:outreg_sys=outreg_auxsys} is illustrated, respectively, on the left and on the right.
    
\begin{figure}
\begin{minipage}{0.49\textwidth}
\centerline{\includegraphics[width=1\textwidth]{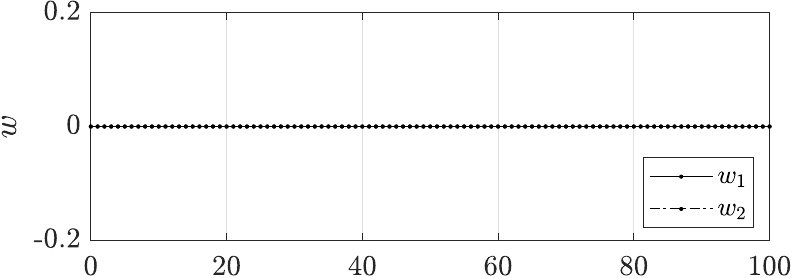}}
\centerline{\hspace*{-0.5mm}\includegraphics[width=1.01\textwidth]{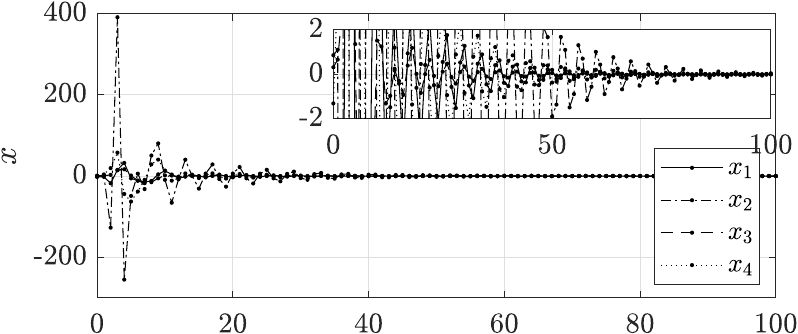}}
\centerline{\hspace*{0.5mm}\includegraphics[width=1\textwidth]{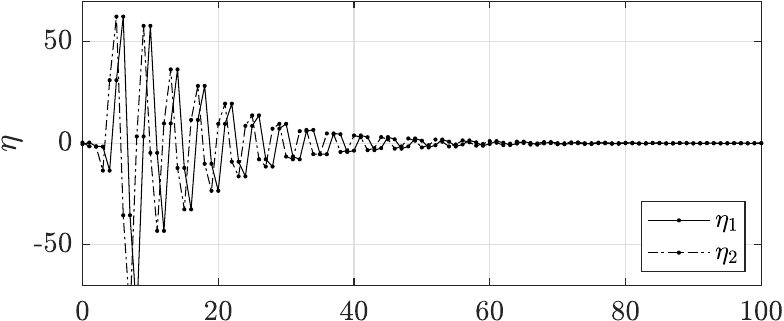}}
\centerline{\hspace*{0.5mm}\includegraphics[width=1\textwidth]{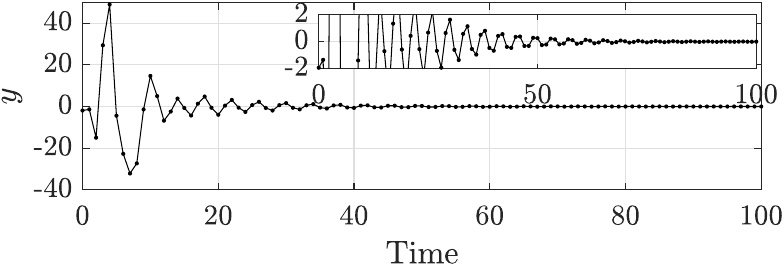}}
\end{minipage}
\begin{minipage}{0.49\textwidth}
\centerline{\includegraphics[width=1\textwidth]{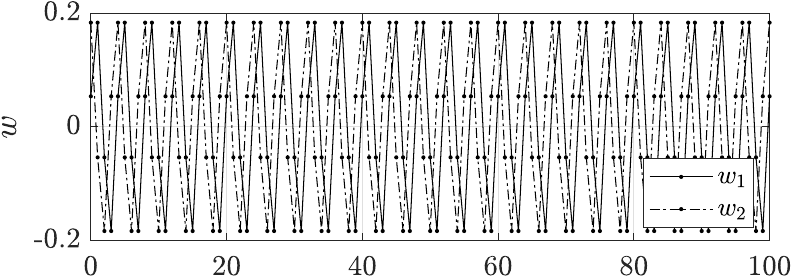}}
\centerline{\hspace*{-0.5mm}\includegraphics[width=1.01\textwidth]{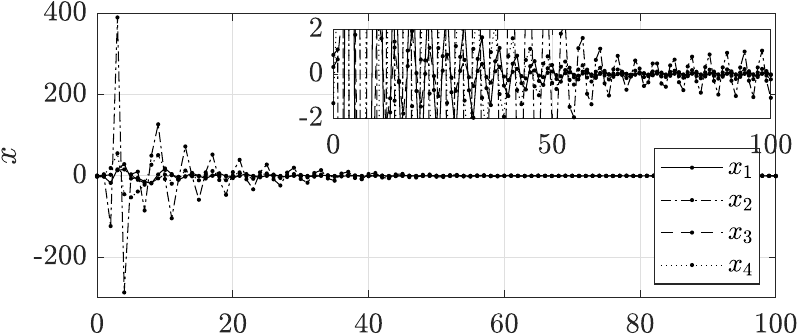}}
\centerline{\hspace*{0.5mm}\includegraphics[width=1\textwidth]{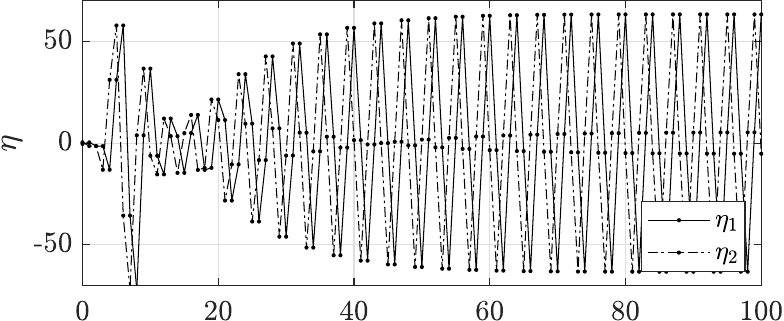}}
\centerline{\hspace*{0.5mm}\includegraphics[width=1\textwidth]{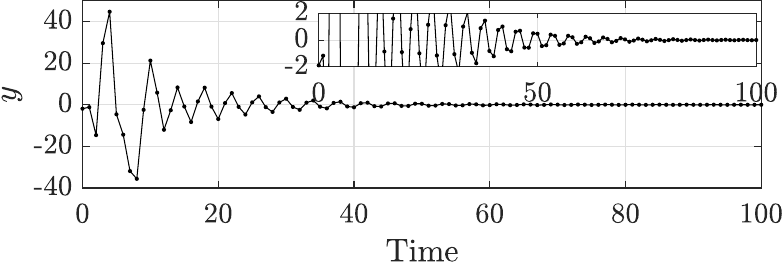}}
\end{minipage}
\caption{Data-based design of a controller achieving output regulation of~\eqref{eq:sys_cl_xy2}. (Left)~Solutions with $w=0$. (Right)~Solutions with $w\neq 0$.}
\label{fig:closed_loop}
\end{figure}

\appendix

\section{Proof of Claim~\ref{claim:1:xy}}
\label{app:claim:1:xy}

The solution $w(\cdot)$ to~\eqref{eq:sys:w} with initial condition $w(0) = \hat{w}$ is $w(k) = S^k \hat{w}$ for all $k \ge 0$ and we can thus consider $\{ w(k) \}_{k=0}^\infty$ as a fictitious input to~\eqref{eq:sys:x}.
The solution $x(\cdot)$ to~\eqref{eq:sys:x} with initial condition $x(0)=\hat{x}$, input $\{ u(k) \}_{k=0}^\infty$ and fictitious input $\{ w(k) \}_{k=0}^\infty$ satisfies, for each $k \ge \ell$,
\begin{align*}
            \smat{
                y(k - \ell)\\
                \vdotsS\\
                y(k - 1)} = \Obs x(k - \ell) + \Tu 
            \smat{
                u(k - \ell)\\
                \vdotsS\\
                u(k - 1)} + \Tw 
			\smat{
                w(k - \ell)\\
                \vdotsS\\
                w(k - 1)},
\end{align*}
which is \eqref{eq:claim1:yellxy:yell}.
\eqref{eq:claim1:yellxy:yell} implies that for each $k \ge \ell$
\begin{align}
\label{eq:claim1:x_k-ell}
x(k - \ell) = \Obs^{\tu{L}} \left(
            \smat{
                y(k - \ell)\\
                \vdotsS\\
                y(k - 1)}
- \Tu \smat{
                u(k - \ell)\\
                \vdotsS\\
                u(k - 1)}
- \Tw \smat{
                w(k - \ell)\\
                \vdotsS\\
                w(k - 1)}
                \right)
\end{align}
since $(A,C)$ is observable and $\Obs$ possesses a left inverse.
Also, the solution $x(\cdot)$ to~\eqref{eq:claim1:yellxy:x} with initial condition $x(0) = \hat{x}$, input $\{ u(k) \}_{k=0}^\infty$ and fictitious input $\{ w(k) \}_{k=0}^\infty$ satisfies, for each $k \ge \ell$,
\begin{align*}
x(k)  = A^\ell x(k-\ell) &  + A^{\ell-1} B u(k-\ell) + \dots + B u(k-1) \\
&  + A^{\ell-1} P w(k-\ell) + \dots + P w(k-1),
\end{align*}
as can be proven by mathematical induction; thus,
\begin{align*}
x(k) & = A^\ell x(k-\ell) + \Ru \smat{u(k-\ell)\\ \vdotsS \\ u(k-1)} + \Rw \smat{w(k-\ell) \\ \vdotsS \\ w(k-1)}\\
& \overset{\eqref{eq:claim1:x_k-ell}}{=} A^\ell \Obs^{\tu{L}} 
\left(
            \smat{
                y(k - \ell)\\
                \vdotsS\\
                y(k - 1)}
- \Tu \smat{
                u(k - \ell)\\
                \vdotsS\\
                u(k - 1)}
- \Tw \smat{
                w(k - \ell)\\
                \vdotsS\\
                w(k - 1)}
\right)\\
& \qquad + \Ru \smat{u(k-\ell)\\ \vdotsS \\ u(k-1)} + \Rw \smat{w(k-\ell) \\ \vdotsS \\ w(k-1)},
\end{align*}
which is \eqref{eq:claim1:yellxy:x}.
\eqref{eq:claim1:yellxy:y} follows from $y(\cdot) = C x(\cdot) + Q w(\cdot)$.

\section{Proof of Lemma~\ref{lemma:data_rel}}
\label{app:lemma:data_rel}

By Claim~\ref{claim:1:xy} and Assumption~\ref{ass:observ}, \eqref{eq:claim1:yellxy:y} holds for each $k = \ell, \dots, T$; so, for each $k = \ell, \dots, T$,
\begin{align*}
& \left[
\begin{smallarray}{c}
y(k-\ell+1)\\
\vdotsS\\
y(k-1)\\
y(k)\\
\hline
u(k-\ell+1)\\
\vdotsS\\
u(k)\\
\hline
\eta(k+1)
\end{smallarray}
\right] \\
& =
\left[
\begin{smallarray}{c}
y(k-\ell+1)\\
\vdotsS\\
y(k-1)\\
CA^{\ell} \Obs^{\tu{L}}
	\smat{
	y(k-\ell)\\
	\vdotsS\\
	y(k-1)            
	}  + (C\Ru - CA^{\ell} \Obs^{\tu{L}}\Tu)
	\smat{
	u(k-\ell)\\
	\vdotsS \\
	u(k-1)
	}  + \smat{C\Rw - CA^{\ell} \Obs^{\tu{L}}\Tw & Q} \smat{w(k-\ell)\\
	\vdotsS \\
	w(k-1)\\ w(k)} \\
\hline 
u(k-\ell+1)\\
\vdotsS\\
u(k)\\
\hline
\Phi \eta(k) + G y(k)
\end{smallarray}
\right]\\
& =
\left[
\begin{smallarray}{c}
(\mathbf{F} + \mathbf{L} Z_1  )
\left[
\begin{smallarray}{c}
	y(k-\ell)\\
	\vdotsS\\
	y(k-1)\\
	\hline
	u(k-\ell)\\
	\vdotsS\\
	u(k-1)\\	
\end{smallarray}
\right]
+
\Lb Z_2 \Sb w(k) + \mathbf{B} u(k)
\\
\Phi \eta(k) + G \left(
Z_1
\left[
\begin{smallarray}{c}
	y(k-\ell)\\
	\vdotsS\\
	y(k-1)\\
	\hline
	u(k-\ell)\\
	\vdotsS\\
	u(k-1)\\	
\end{smallarray}
\right]
+
Z_2 \Sb w(k)
\right)
\end{smallarray}
\right]\\
& = \underbrace{\bmat{\mathbf{F}+ \Lb Z_1 & 0\\ G Z_1 & \Phi}}_{=\bar{\Ab}}
\left[
\begin{smallarray}{c}
	y(k-\ell)\\
	\vdotsS\\
	y(k-1)\\
	u(k-\ell)\\
	\vdotsS\\
	u(k-1)\\	
	\hline
	\eta(k)
\end{smallarray}
\right]
+
\underbrace{\bmat{\mathbf{B}\\ 0}}_{=\bar{\Bb}} u(k) + \underbrace{\bmat{\Lb Z_2\\ G Z_2}}_{=\bar{\Pb}} \Sb w(k)
\end{align*}
where we have used $\Sb$ thanks to Assumption~\ref{ass:S} and the definitions of $\bar{\Ab}$, $\bar{\Bb}$, $\bar{\Pb}$ in~\eqref{eq:sys:xi_eta:barAbarB}.
By stacking the so-obtained columns and forming the matrices $U_1$, $\Psi_0$, $\Psi_1$, $W_0$ in~\eqref{eq:data}-\eqref{eq:W0}, \eqref{eq:data-matrices-identity} follows.
 
\section{Proof of Lemma~\ref{lemma:factorization_W0}}
\label{app:lemma:factorization_W0}

The factorization of $W_0$ in Lemma~\ref{lemma:factorization_W0} can be obtained from the next claim, which is the discrete-time counterpart of \cite[Lemma 1]{hu2024dd-output-regulation}. 

\begin{claim}
    \label{claim:exo_response}
Consider the system
\[
w^+= Sw
\]
with initial condition $w(0)=\hat{w} \in \R^{n_w}$. 
Let $\lambda_1$, \dots, $\lambda_r$ be the real eigenvalues of $S$ and $\rho_{r+1} \e^{\i \theta_{r+1}}$, $\rho_{r+1} \e^{-\i \theta_{r+1}}$, \dots, $\rho_{r+s} \e^{\i \theta_{r+s}}$, $\rho_{r+s} \e^{-\i \theta_{r+s}}$ be the nonreal eigenvalues of $S$.
Let $J$ be the Jordan form of $S$, namely, 
\begin{align}
J & = J_{k_1}(\lambda_1) \oplus \dots \oplus J_{k_r}(\lambda_r) \oplus J_{k_{r+1}}(\rho_{r+1} \e^{\i \theta_{r+1}}) \oplus J_{k_{r+1}}(\rho_{r+1} \e^{-\i \theta_{r+1}}) \label{jordan_can_form} \\
&  \hspace*{33mm} \oplus \dots \oplus J_{k_{r+s}}(\rho_{r+s} \e^{\i \theta_{r+s}}) \oplus J_{k_{r+s}}(\rho_{r+s} \e^{-\i \theta_{r+s}})
\notag 
\end{align}
where, for any $\mu \in \C$ eigenvalue of $S$,
\begin{align}
J_k(\mu) := \smat{
\mu & 1 & 0 & \dots & 0\\ 
0 & \mu & 1 & \dots & 0\\ 
\vdotsS & \vdotsS & \ddotsS & \ddotsS & \vdotsS\\ 
0 & 0 & \dots & \mu & 1\\
0 & 0 & \dots & 0 & \mu
} 
\in \C^{k \times k}
\label{jordan_block}
\end{align}
is the Jordan block of order $k$ associated with the eigenvalue $\mu$ \cite[Def.~3.1.1]{horn2013matrix} so that $k_1 + \dots + k_r + 2 k_{r+1} + \dots + 2 k_{r+s} = n_w$.

Then, there exist a nonsingular matrix $\tilde{\mathcal{T}}$, a matrix $L(\hat{w}) \in \R^{n_w \times n_w}$, matrices 
\begin{subequations}
\label{M-realORcomplex}
\begin{align}
M_{k_i}(t, \lambda_i) & \!:= \!
\smat{
{t \choose t-k_i+1} \lambda_i^{t-k_i+1} \\
{t \choose t-k_i+2} \lambda_i^{t-k_i+2} \\
\vdotsS \\
{t \choose t-1} \lambda_i^{t-1}\\
\lambda_i^t
}\!, \,\, i = 1, \dots, r, \, t\in \Z_{\ge 0} \label{M-real}
\\
M_{k_i}(t,\rho_i,\theta_i) & \!:=\! \left[
\begin{smallarray}{c}
{t \choose t-k_i+1} \rho_i^{t-k_i+1} \cos(\theta_i (t-k_i+1)) \\
{t \choose t-k_i+1} \rho_i^{t-k_i+1} \sin(\theta_i (t-k_i+1)) \\ \hline
\vdotsS \\ \hline
{t \choose t-1} \rho_i^{t-1} \cos(\theta_i (t-1)) \\
{t \choose t-1} \rho_i^{t-1} \sin(\theta_i (t-1)) \\ \hline
\rho_i^t \cos(\theta_i t) \\
\rho_i^t \sin(\theta_i t)
\end{smallarray}
\right]\!, \label{M-complex} \, i = r+1, \dots, r+s, \, t\in \Z_{\ge 0}
\end{align}
\end{subequations}
such that, for all $t \in \Z_{\ge 0}$,
\begin{align}
& w(t) = S^t w(0) = S^t \hat{w} = \tilde{\mathcal{T}} L(\hat{w}) \smat{M_{k_1}(t,\lambda_1)\\ \vdotsS \\ M_{k_r}(t,\lambda_r) \\ M_{k_{r+1}}(t,\rho_{r+1},\theta_{r+1})\\ \vdotsS \\ M_{k_{r+s}}(t,\rho_{r+s},\theta_{r+s})} =: \tilde{\mathcal{T}}  L(\hat{w}) M(t). \label{w=TLM}
\end{align}
\end{claim}
\begin{proof}
By the Jordan canonical form theorem \cite[Thm.~3.1.11]{horn2013matrix}, there is a nonsingular matrix $\mathcal{T}$ such that, for $J$ in \eqref{jordan_can_form}, $J = \mathcal{T}^{-1} S \mathcal{T}$.
For $t \in \Z_{\ge 0}$ and any $\mu \in \C$ eigenvalue of $S$, the $t$-th power of the Jordan block $J_{k}(\mu)$ in~\eqref{jordan_block} is
\begin{align}
\label{jordan_block_power}
J_{k}(\mu)^t = 
\smat{\mu^t & {t \choose t-1} \mu^{t-1} & {t \choose t-2} \mu^{t-2} & \dots & {t \choose t-k+1} \mu^{t-k+1}\\
0 & \mu^t & {t \choose t-1} \mu^{t-1} & \dots & {t \choose t-k+2} \mu^{t-k+2}\\
\vdotsS & \vdotsS & \vdotsS & \ddotsS & \vdotsS \\
0 & 0 & 0 & \dots & {t \choose t-1} \mu^{t-1}\\
0 & 0 & 0 & \dots & \mu^t
 }.
\end{align}
We prove this by mathematical induction.
As for the base case,
\begin{align*}
J_{k}(\mu)^0 = 
\smat{\mu^0 & {0 \choose -1} \mu^{-1} & {0 \choose -2} \mu^{-2} & \dots & {0 \choose -k+1} \mu^{-k+1}\\
0 & \mu^0 & {0 \choose -1} \mu^{-1} & \dots & {0 \choose -k+2} \mu^{-k+2}\\
\vdotsS & \vdotsS & \vdotsS & \ddotsS & \vdotsS \\
0 & 0 & 0 & \dots & { 0\choose -1} \mu^{-1}\\
0 & 0 & 0 & \dots & \mu^0
 } = I
\end{align*}
where we use the definition ${0 \choose -1}={0 \choose -2} = \dots = {0 \choose -k+1}= 0$, see notation in Section~\ref{sec:notation}.
As for the induction step, we assume that \eqref{jordan_block_power} holds and need to show that
\begin{align}
\label{jordan_block_power_to_be_proven}
J_{k}(\mu)^{t+1} = 
\smat{\mu^{t+1} & {t+1 \choose t} \mu^{t} & {t+1 \choose t-1} \mu^{t-1} & \dots & {t+1 \choose t-k+2} \mu^{t-k+2}\\
0 & \mu^{t+1} & {t+1 \choose t} \mu^{t} & \dots & {t+1 \choose t-k+3} \mu^{t-k+3}\\
\vdotsS & \vdotsS & \vdotsS & \ddotsS & \vdotsS \\
0 & 0 & 0 & \dots & {t+1 \choose t} \mu^{t}\\
0 & 0 & 0 & \dots & \mu^{t+1}
 }.
\end{align}
We have
\begin{align}
\label{jordan_block_power_what_we_prove}
& J_{k}(\mu)^{t+1} = J_{k}(\mu)^t J_{k}(\mu) \\ 
& \overset{\eqref{jordan_block_power}}{=} 
\smat{\mu^t & {t \choose t-1} \mu^{t-1} & {t \choose t-2} \mu^{t-2} & \dots & {t \choose t-k+1} \mu^{t-k+1}\\
0 & \mu^t & {t \choose t-1} \mu^{t-1} & \dots & {t \choose t-k+2} \mu^{t-k+2}\\
\vdotsS & \vdotsS & \vdotsS & \ddotsS & \vdotsS \\
0 & 0 & 0 & \dots & {t \choose t-1} \mu^{t-1}\\
0 & 0 & 0 & \dots & \mu^t
}
\smat{
\mu & 1 & 0 & \dots & 0\\ 
0 & \mu & 1 & \dots & 0\\ 
\vdotsS & \vdotsS & \ddotsS & \ddotsS & \vdotsS\\ 
0 & 0 & \dots & \mu & 1\\
0 & 0 & \dots & 0 & \mu
} \notag \\
& = \smat{\mu^{t+1} & \mu^{t} + {t \choose t-1} \mu^{t} & {t \choose t-1} \mu^{t-1} + {t \choose t-2} \mu^{t-1} & \dots & {t \choose t-k+2} \mu^{t-k+2} + {t \choose t-k+1} \mu^{t-k+2}\\
0 & \mu^{t+1} & \mu^{t} + {t \choose t-1} \mu^{t} & \dots & {t \choose t-k+3} \mu^{t-k+3} + {t \choose t-k+2} \mu^{t-k+3}\\
\vdotsS & \vdotsS & \vdotsS & \ddotsS & \vdotsS \\
0 & 0 & 0 & \dots &  \mu^{t} + {t \choose t-1} \mu^{t}\\
0 & 0 & 0 & \dots & \mu^{t+1}
 } \notag \\
 & = \smat{\mu^{t+1} & {t+1 \choose t} \mu^{t} & {t+1 \choose t-1} \mu^{t-1} & \dots & {t+1 \choose t-k+2} \mu^{t-k+2}\\
0 & \mu^{t+1} & {t+1 \choose t} \mu^{t} & \dots & {t+1 \choose t-k+3} \mu^{t-k+3}\\
\vdotsS & \vdotsS & \vdotsS & \ddotsS & \vdotsS \\
0 & 0 & 0 & \dots & {t+1 \choose t} \mu^{t}\\
0 & 0 & 0 & \dots & \mu^{t+1} 
} \notag
\end{align}
where in the last equality we used that for all integers $p$ and $q$ such that $p \ge q$ and $p \ge 1$, ${p \choose q} = {p-1 \choose q-1} + {p-1 \choose q}$. (Indeed, this identity is known to hold for all integers $p$, $q$ such that $p \ge q \ge 1$; if $p \ge q$, $p \ge 1$ and $q=0$, ${p \choose 0} =1= {p-1 \choose -1} + {p-1 \choose 0} $; if $p \ge q$, $p \ge 1$ and $q \le -1$, ${p \choose q} = 0 =  {p-1 \choose q-1} + {p-1 \choose q}$.)
Since \eqref{jordan_block_power_to_be_proven} and \eqref{jordan_block_power_what_we_prove} coincide, we have proven also the induction step.
By specializing \eqref{jordan_block_power}, we have that for $t \in \Z_{\ge 0}$ and the nonreal eigenvalues $\rho_i \e^{\i \theta_i}$ and $\rho_i \e^{-\i \theta_i}$, with $i = r+1, \dots, r+s$,
\begin{align}
\label{jordan_block_power_nonreal}
& J_{k_i}(\rho_i \e^{\i \theta_i})^t \\
& = 
\smat{
\rho_i^t \e^{\i \theta_i t} & {t \choose t-1} \rho_i^{t-1} \e^{\i \theta_i(t-1)} & {t \choose t-2} \rho_i^{t-2} \e^{\i \theta_i (t-2)} & \dots & {t \choose t-k_i+1} \rho_i^{t-k_i+1} \e^{\i \theta_i(t-k_i+1)} \\
0 & \rho_i^t \e^{\i \theta_i t} & {t \choose t-1} \rho_i^{t-1} \e^{\i \theta_i(t-1)} & \dots & {t \choose t-k_i+2} \rho_i^{t-k_i+2} \e^{\i \theta_i(t-k_i+2)} \\
0 & 0 & \rho_i^t \e^{\i \theta_i t} &  \dots & {t \choose t-k_i+3} \rho_i^{t-k_i+3} \e^{\i \theta_i(t-k_i+3)} \\
\vdotsS & \vdotsS & \vdotsS & \ddotsS & \vdotsS \\
0 & 0 & 0 & \dots & {t \choose t-1} \rho_i^{t-1} \e^{\i \theta_i(t-1)} \\
0 & 0 & 0 & \dots & \rho_i^t \e^{\i \theta_i t}
 },\!\!\! \notag
\end{align}
so that $J_{k_i}(\rho_i \e^{-\i \theta_i})^t = \overline{J_{k_i}(\rho_i \e^{\i \theta_i})^t}$.
With a similar approach to that used to obtain the real Jordan form \cite[\S 3.4]{horn2013matrix}, we apply two similarity transformations to the blocks
$J_{k_i}(\rho_i \e^{\i \theta_i})^t \oplus J_{k_i}(\rho_i \e^{-\i \theta_i})^t$, $i = r+1, \dots, r+s$ contained in $J^t$, see~\eqref{jordan_can_form}.
(i)~To transform $J_{k_i}(\rho_i \e^{\i \theta_i})^t \oplus J_{k_i}(\rho_i \e^{-\i \theta_i})^t$ into a matrix with real entries we consider
\begin{align*}
J^{\textup{cc}}_{2 k_i}(t,\rho_i,\theta_i) & := \bmat{-\i I_{k_i} & -\i I_{k_i}\\ I_{k_i} & -I_{k_i}} \bmat{J_{k_i}(\rho_i \e^{\i \theta_i})^t & 0\\ 0 & J_{k_i}(\rho_i \e^{-\i \theta_i})^t} \bmat{-\i I_{k_i} & -\i I_{k_i}\\ I_{k_i} & -I_{k_i}}^{-1}\\
& = \bmat{-\i J_{k_i}(\rho_i \e^{\i \theta_i})^t & -\i \overline{J_{k_i}(\rho_i \e^{\i \theta_i})^t} \\ J_{k_i}(\rho_i \e^{\i \theta_i})^t &  -\overline{J_{k_i}(\rho_i \e^{\i \theta_i})^t}} \bmat{-I_{k_i} & \i I_{k_i}\\ -I_{k_i} & -\i I_{k_i} } \frac{1}{2\i}\\
& = \bmat{
\frac{J_{k_i}(\rho_i \e^{\i \theta_i})^t + \overline{J_{k_i}(\rho_i \e^{\i \theta_i})^t}}{2} & \frac{J_{k_i}(\rho_i \e^{\i \theta_i})^t - \overline{J_{k_i}(\rho_i \e^{\i \theta_i})^t}}{2\i} \\ 
-\frac{J_{k_i}(\rho_i \e^{\i \theta_i})^t - \overline{J_{k_i}(\rho_i \e^{\i \theta_i})^t}}{2\i} & \frac{J_{k_i}(\rho_i \e^{\i \theta_i})^t + \overline{J_{k_i}(\rho_i \e^{\i \theta_i})^t}}{2} }
\end{align*}
where, by \eqref{jordan_block_power_nonreal},
\begin{align*}
&  \frac{J_{k_i}(\rho_i \e^{\i \theta_i})^t + \overline{J_{k_i}(\rho_i \e^{\i \theta_i})^t}}{2} \\
& = \smat{
\rho_i^t \cos(\theta_i t) & {t \choose t-1} \rho_i^{t-1} \cos(\theta_i(t-1)) & {t \choose t-2} \rho_i^{t-2} \cos(\theta_i(t-2)) & \dots & {t \choose t-k_i+1} \rho_i^{t-k_i+1} \cos(\theta_i(t-k_i+1)) \\
0 & \rho_i^t \cos(\theta_i t) & {t \choose t-1} \rho_i^{t-1} \cos(\theta_i(t-1)) & \dots & {t \choose t-k_i+2} \rho_i^{t-k_i+2} \cos( \theta_i(t-k_i+2))\\
0 & 0 & \rho_i^t \cos(\theta_i t) & \dots & {t \choose t-k_i+3} \rho_i^{t-k_i+3} \cos( \theta_i(t-k_i+3))\\
\vdotsS & \vdotsS & \vdotsS & \ddotsS & \vdotsS \\
0 & 0 & 0 & \dots & {t \choose t-1} \rho_i^{t-1} \cos(\theta_i(t-1)) \\
0 & 0 & 0 & \dots & \rho_i^t \cos(\theta_i t)
 }
 \\
& \frac{J_{k_i}(\rho_i \e^{\i \theta_i})^t - \overline{J_{k_i}(\rho_i \e^{\i \theta_i})^t}}{2 \i} \\
& = 
\smat{
\rho_i^t \sin(\theta_i t) & {t \choose t-1} \rho_i^{t-1} \sin(\theta_i(t-1)) & {t \choose t-2} \rho_i^{t-2} \sin(\theta_i(t-2)) & \dots & {t \choose t-k_i+1} \rho_i^{t-k_i+1} \sin(\theta_i(t-k_i+1)) \\
0 & \rho_i^t \sin(\theta_i t) & {t \choose t-1} \rho_i^{t-1} \sin(\theta_i(t-1)) & \dots & {t \choose t-k_i+2} \rho_i^{t-k_i+2} \sin( \theta_i(t-k_i+2))\\
0 & 0 & \rho_i^t \sin(\theta_i t) & \dots & {t \choose t-k_i+3} \rho_i^{t-k_i+3} \sin( \theta_i(t-k_i+3))\\
\vdotsS & \vdotsS & \vdotsS & \ddotsS & \vdotsS \\
0 & 0 & 0 & \dots & {t \choose t-1} \rho_i^{t-1} \sin(\theta_i(t-1)) \\
0 & 0 & 0 & \dots & \rho_i^t \sin(\theta_i t)
 }.
\end{align*}
(ii)~To recover a block triangular structure in $J^{\textup{cc}}_{2 k_i}(t,\rho_i,\theta_i)$, we consider row and column permutations as
\begin{align}
& 
\begin{matrix}
\overbrace{\rule{45pt}{0pt}}^{k_i\text{ columns }}\overbrace{\rule{45pt}{0pt}}^{k_i\text{ columns }}\hspace*{180pt}
\\
\left[
\begin{array}{c|c}
\smatNoB{1 & 0 & \dots & 0 & 0\\
0 & 0 & \dots & 0 & 0} & \smatNoB{0 & 0 & \dots & 0 & 0\\
1 & 0 & \dots & 0 & 0}\\ \hline
\smatNoB{0 & 1 & \dots & 0 & 0\\
0 & 0 & \dots & 0 & 0} & \smatNoB{0 & 0 & \dots & 0 & 0\\
0 & 1 & \dots & 0 & 0}\\ \hline
\vdotsS & \vdotsS \\ \hline
\smatNoB{0 & 0 & \dots & 1 & 0\\
0 & 0 & \dots & 0 & 0} & \smatNoB{0 & 0 & \dots & 0 & 0\\
0 & 0 & \dots & 1 & 0}\\ \hline
\smatNoB{0 & 0 & \dots & 0 & 1\\
0 & 0 & \dots & 0 & 0} & \smatNoB{0 & 0 & \dots & 0 & 0\\
0 & 0 & \dots & 0 & 1}
\end{array}
\right]
J^{\textup{cc}}_{2 k_i}(t,\rho_i,\theta_i)
\left[
\begin{array}{c|c}
\smatNoB{1 & 0 & \dots & 0 & 0\\
0 & 0 & \dots & 0 & 0} & \smatNoB{0 & 0 & \dots & 0 & 0\\
1 & 0 & \dots & 0 & 0}\\ \hline
\smatNoB{0 & 1 & \dots & 0 & 0\\
0 & 0 & \dots & 0 & 0} & \smatNoB{0 & 0 & \dots & 0 & 0\\
0 & 1 & \dots & 0 & 0}\\ \hline
\vdotsS & \vdotsS \\ \hline
\smatNoB{0 & 0 & \dots & 1 & 0\\
0 & 0 & \dots & 0 & 0} & \smatNoB{0 & 0 & \dots & 0 & 0\\
0 & 0 & \dots & 1 & 0}\\ \hline
\smatNoB{0 & 0 & \dots & 0 & 1\\
0 & 0 & \dots & 0 & 0} & \smatNoB{0 & 0 & \dots & 0 & 0\\
0 & 0 & \dots & 0 & 1}
\end{array}
\right]^{-1}
\end{matrix} \notag \\
& 
= 
\left[
\begin{smallarray}{cc|cc|c}
\rho_i^t \cos(\theta_i t) & \rho_i^t \sin(\theta_i t) & {t \choose t-1} \rho_i^{t-1} \cos(\theta_i (t-1)) & {t \choose t-1} \rho_i^{t-1} \sin(\theta_i (t-1)) & \dots \\
-\rho_i^t \sin(\theta_i t) & \rho_i^t \cos(\theta_i t) &  -{t \choose t-1} \rho_i^{t-1} \sin(\theta_i (t-1)) &  {t \choose t-1} \rho_i^{t-1} \cos(\theta_i (t-1)) & \dots \\ \hline
0 & 0 & \rho_i^t \cos(\theta_i t) & \rho_i^t \sin(\theta_i t) & \dots \\
0 & 0 & -\rho_i^t \sin(\theta_i t) & \rho_i^t \cos(\theta_i t) & \dots \\ \hline
\vdotsS & \vdotsS & \vdotsS & \vdotsS & \\ \hline
0 & 0 & 0 & 0 &  \dots \\
0 & 0 & 0 & 0 & \dots \\\hline
0 & 0 & 0 & 0 & \dots \\
0 & 0 & 0 & 0 & \dots
\end{smallarray}
\right. \notag \\
& \qquad
\left.
\begin{smallarray}{c|cc}
\dots & {t \choose t-k_i+1} \rho_i^{t-k_i+1} \cos(\theta_i (t-k_i+1)) & {t \choose t-k_i+1} \rho_i^{t-k_i+1} \sin(\theta_i (t-k_i+1))\\
\dots & -{t \choose t-k_i+1} \rho_i^{t-k_i+1} \sin(\theta_i (t-k_i+1)) & {t \choose t-k_i+1} \rho_i^{t-k_i+1} \cos(\theta_i (t-k_i+1)) \\ \hline
\dots & {t \choose t-k_i+2} \rho_i^{t-k_i+2} \cos(\theta_i (t-k_i+2)) & {t \choose t-k_i+2} \rho_i^{t-k_i+2} \sin(\theta_i (t-k_i+2)) \\
\dots & -{t \choose t-k_i+2} \rho_i^{t-k_i+2} \sin(\theta_i (t-k_i+2)) & {t \choose t-k_i+2} \rho_i^{t-k_i+2} \cos(\theta_i (t-k_i+2))\\ \hline
 & \vdotsS & \vdotsS\\ \hline
\dots &{t \choose t-1}  \rho_i^{t-1} \cos(\theta_i (t-1)) & {t \choose t-1} \rho_i^{t-1} \sin(\theta_i (t-1))\\
\dots & -{t \choose t-1} \rho_i^{t-1} \sin(\theta_i (t-1)) & {t \choose t-1} \rho_i^{t-1} \cos(\theta_i (t-1))\\\hline
\dots & \rho_i^t \cos(\theta_i t) & \rho_i^t \sin(\theta_i t)\\
\dots & -\rho_i^t \sin(\theta_i t) & \rho_i^t \cos(\theta_i t)
\end{smallarray}
\right] \notag \\
& =: \hat{J}^{\textup{cc}}_{2 k_i}(t,\rho_i,\theta_i) \label{jordan_block_power_nonreal_realified}
\end{align}
where we used that the inverse of a permutation matrix is its transpose.
In summary, the similarity transformations in (i) and (ii) yield a similarity matrix $\hat{\mathcal{T}}$ such that, from \eqref{jordan_block_power} and \eqref{jordan_block_power_nonreal_realified}, for each $t \in \Z_{\ge 0}$,
\begin{align*}
& \hat{J}(t) := \hat{\mathcal{T}}^{-1} J^t \hat{\mathcal{T}} = 
J_{k_1}(\lambda_1)^t \oplus \dots \oplus J_{k_r}(\lambda_r)^t \\
& \qquad \oplus 
\hat{J}^{\textup{cc}}_{2 k_{r+1}}(t,\rho_{r+1},\theta_{r+1}) \oplus \dots \oplus
\hat{J}^{\textup{cc}}_{2 k_{r+s}}(t,\rho_{r+s},\theta_{r+s}).
\end{align*}  
We are now in the position to obtain the solution $w(\cdot)$ with initial condition $w(0)=\hat{w}$ to $w^+ = S w$.
We consider the change of coordinates $\zeta = \hat{\mathcal{T}}^{-1} \mathcal{T}^{-1} w$ and $\hat{\zeta} := \hat{\mathcal{T}}^{-1} \mathcal{T}^{-1} \hat{w}$.
Then, for each $t \in \Z_{\ge 0}$,
\begin{align*}
w(t) & = S^t w(0) = (\mathcal{T} J \mathcal{T}^{-1})^t \hat{w} = \mathcal{T} J^t \mathcal{T}^{-1} \hat{w} = \mathcal{T} \hat{\mathcal{T}} \hat{J}(t) \hat{\mathcal{T}}^{-1} \mathcal{T}^{-1} \hat{w}  = \mathcal{T} \hat{\mathcal{T}} \hat{J}(t) \hat{\zeta}.
\end{align*}
Partition $\zeta$ as $\zeta = (\zeta^1, \dots, \zeta^r, \zeta^{r+1}, \dots, \zeta^{r+s})$ where $\zeta^1 \in \R^{k_1}$, \dots, $\zeta^r \in \R^{k_r}$, $\zeta^{r+1}\in \R^{2 k_{r+1}}$, \dots, $\zeta^{r+s} \in \R^{2 k_{r+s}}$.
Then, for each $t \in \Z_{\ge 0}$,
\begin{align*}
\zeta(t) = 
\smat{
\zeta^1(t)\\ \vdotsS \\ \zeta^r(t) \\ \zeta^{r+1}(t) \\ \vdotsS \\ \zeta^{r+s}(t)
}
=\hat{\mathcal{T}}^{-1} \mathcal{T}^{-1} w(t) = \hat{J}(t) \hat{\zeta}
=
\smat{
J_{k_1}(\lambda_1)^t \hat{\zeta}^1 \\ \vdotsS \\ J_{k_r}(\lambda_r)^t \hat{\zeta}^r \\ \hat{J}^{\textup{cc}}_{2 k_{r+1}}(t,\rho_{r+1},\theta_{r+1}) \hat{\zeta}^{r+1} \\ \vdotsS \\ \hat{J}^{\textup{cc}}_{2 k_{r+s}}(t,\rho_{r+s},\theta_{r+s}) \hat{\zeta}^{r+s}
}
\end{align*}
and for $i = 1, \dots, r$,
\begin{align*}
& J_{k_i}(\lambda_i)^t \hat{\zeta}^i \overset{\eqref{jordan_block_power}}{=}
\smat{\lambda_i^t & {t \choose t-1} \lambda_i^{t-1} & {t \choose t-2} \lambda_i^{t-2} & \dots & {t \choose t-k_i+1} \lambda_i^{t-k_i+1}\\
0 & \lambda_i^t & {t \choose t-1} \lambda_i^{t-1} & \dots & {t \choose t-k_i+2}\lambda_i^{t-k_i+2}\\
\vdotsS & \vdotsS & \vdotsS & \ddotsS & \vdotsS \\
0 & 0 & 0 & \dots & {t \choose t-1} \lambda_i^{t-1}\\
0 & 0 & 0 & \dots & \lambda_i^t}
\smat{\hat{\zeta}^i_1\\ \hat{\zeta}^i_2\\ \hat{\zeta}^i_3\\  \vdotsS \\ \hat{\zeta}^i_{k_i}} \\
& = 
\smat{
\hat{\zeta}^i_{k_i} & \hat{\zeta}^i_{k_i-1} & \dots & \hat{\zeta}^i_{2} & \hat{\zeta}^i_{1}\\ 
0 & \hat{\zeta}^i_{k_i} & \dots & \hat{\zeta}^i_{3} & \hat{\zeta}^i_{2}\\ 
\vdotsS & \vdotsS & \ddotsS & \vdotsS & \vdotsS \\
0 & 0 & \dots & \hat{\zeta}^i_{k_i} & \hat{\zeta}^i_{k_i-1}\\
0 & 0 & \dots & 0 & \hat{\zeta}^i_{k_i}
}
\smat{
{t \choose t-k_i+1} \lambda_i^{t-k_i+1} \\
{t \choose t-k_i+2} \lambda_i^{t-k_i+2} \\
\vdotsS \\
{t \choose t-1} \lambda_i^{t-1}\\
\lambda_i^t
}
\overset{\eqref{M-real}}{=:} L_{k_i}(\hat{w}) M_{k_i}(t,\lambda_i)
\end{align*}
and for $i = r+1, \dots, r+s$,
\begin{align*}
& \hat{J}^{\textup{cc}}_{2 k_i}(t,\rho_i,\theta_i) \hat{\zeta}^i  = \hat{J}^{\textup{cc}}_{2 k_i}(t,\rho_i,\theta_i) \left[
\begin{smallarray}{c}
\bmat{\hat{\zeta}^i_{1\textup{c}}\\ \hat{\zeta}^i_{1\textup{s}}} \\ \hline \bmat{\hat{\zeta}^i_{2\textup{c}}\\ \hat{\zeta}^i_{2\textup{s}}} \\ \hline
\vdotsS \\ \hline
\bmat{\hat{\zeta}^i_{k_i\textup{c}}\\ \hat{\zeta}^i_{k_i\textup{s}}}
\end{smallarray}
\right] \\
& \overset{\eqref{jordan_block_power_nonreal_realified}}{=} 
\left[
\begin{smallarray}{cc|c|cc|cc}
\hat{\zeta}^i_{k_i\textup{c}} & \hat{\zeta}^i_{k_i\textup{s}} & \dots & \hat{\zeta}^i_{2\textup{c}} & \hat{\zeta}^i_{2\textup{s}} & \hat{\zeta}^i_{1\textup{c}} & \hat{\zeta}^i_{1\textup{s}}\\
\hat{\zeta}^i_{k_i\textup{s}} & -\hat{\zeta}^i_{k_i\textup{c}} &  \dots &  \hat{\zeta}^i_{2\textup{s}} & -\hat{\zeta}^i_{2\textup{c}} & \hat{\zeta}^i_{1\textup{s}} & -\hat{\zeta}^i_{1\textup{c}} \\ \hline
0 & 0 & \dots & \hat{\zeta}^i_{3\textup{c}} & \hat{\zeta}^i_{3\textup{s}} & \hat{\zeta}^i_{2\textup{c}} & \hat{\zeta}^i_{2\textup{s}} \\
0 & 0 & \dots & \hat{\zeta}^i_{3\textup{s}} & -\hat{\zeta}^i_{3\textup{c}} & \hat{\zeta}^i_{2\textup{s}} & -\hat{\zeta}^i_{2\textup{c}}\\ \hline
\vdotsS & \vdotsS & \ddotsS & \vdotsS & & \vdotsS & \vdotsS \\ \hline
0 & 0 & \dots & 0 &  0 & \hat{\zeta}^i_{k_i\textup{c}} & \hat{\zeta}^i_{k_i\textup{s}}\\
0 & 0 & \dots & 0 & 0 & \hat{\zeta}^i_{k_i\textup{s}} & -\hat{\zeta}^i_{k_i\textup{c}}
\end{smallarray}
\right]
\left[
\begin{smallarray}{c}
{t \choose t-k_i+1} \rho_i^{t-k_i+1} \cos(\theta_i (t-k_i+1)) \\
{t \choose t-k_i+1} \rho_i^{t-k_i+1} \sin(\theta_i (t-k_i+1)) \\ \hline
\vdotsS \\ \hline
{t \choose t-1} \rho_i^{t-1} \cos(\theta_i (t-1)) \\
{t \choose t-1} \rho_i^{t-1} \sin(\theta_i (t-1)) \\ \hline
\rho_i^t \cos(\theta_i t) \\
\rho_i^t \sin(\theta_i t)
\end{smallarray}
\right]\\
& \overset{\eqref{M-complex}}{=:} L_{k_i}(\hat{w}) M_{k_i}(t,\rho_i,\theta_i).
\end{align*}
Note that above we substituted the definitions \eqref{M-real}-\eqref{M-complex} for $M_{k_i}$, $i=1, \dots, r+s$, and $\hat{\zeta} = \hat{\mathcal{T}}^{-1} \mathcal{T}^{-1} \hat{w}$ allows us to define the matrices $\hat{w} \mapsto L_{k_i}(\hat{w})$, $i = 1, \dots, r+s$, as functions of $\hat{w}$.
Hence, for each $t \in \Z_{\ge 0}$ we obtain%
\begin{align*}
& \! w(t) = \mathcal{T} \hat{\mathcal{T}} \zeta(t) = \mathcal{T} \hat{\mathcal{T}} 
\smat{
L_{k_1}(\hat{w}) M_{k_1}(t, \lambda_1)\\
\vdotsS\\
L_{k_r}(\hat{w}) M_{k_r}(t, \lambda_r)\\
L_{k_{r+1}}(\hat{w}) M_{k_{r+1}}(t,\rho_{r+1},\theta_{r+1})\\
\vdotsS\\
L_{k_{r+s}}(\hat{w}) M_{k_{r+s}}(t,\rho_{r+s},\theta_{r+s})\\
}\\
& \! = \mathcal{T} \hat{\mathcal{T}} 
\big(L_{k_1}(\hat{w}) \oplus \dots  \oplus L_{k_r}(\hat{w}) \oplus L_{k_{r+1}}(\hat{w}) \oplus \dots \oplus L_{k_{r+s}}(\hat{w}) \big) \!\! \smat{M_{k_1}(t,\lambda_1)\\ \vdotsS \\ M_{k_r}(t,\lambda_r) \\ M_{k_{r+1}}(t,\rho_{r+1},\theta_{r+1})\\ \vdotsS \\ M_{k_{r+s}}(t,\rho_{r+s},\theta_{r+s})}
\end{align*}
and this coincides with~\eqref{w=TLM} by defining $\tilde{\mathcal{T}} := \mathcal{T} \hat{\mathcal{T}}$, $L(\hat{w}) = L_{k_1}(\hat{w}) \oplus \dots \oplus L_{k_r}(\hat{w}) \oplus L_{k_{r+1}}(\hat{w}) \oplus \dots \oplus L_{k_{r+s}}(\hat{w})$ and by the definition of $M(\cdot)$ in~\eqref{w=TLM}.
\end{proof}

With Claim~\ref{claim:exo_response}, we can prove Lemma~\ref{lemma:factorization_W0}.
Let
\begin{align}
& \mathbf{M} := \bmat{M(\ell) & M(\ell+1) & \dots & M(T)} \in \R^{n_w \times T - \ell+1} \label{boldM} \\
& =
\smat{
M_{k_1}(\ell,\lambda_1) & M_{k_1}(\ell+1,\lambda_1) 
& \dots  & M_{k_1}(T,\lambda_1)
\\ 
\vdotsS & \vdotsS &  & \vdotsS\\ 
M_{k_r}(\ell,\lambda_r) & M_{k_r}(\ell+1,\lambda_r)  &  \ldots & M_{k_r}(T,\lambda_r) \\ 
M_{k_{r+1}}(\ell,\rho_{r+1},\theta_{r+1}) & M_{k_{r+1}}(\ell+1,\rho_{r+1},\theta_{r+1}) &
\dots & M_{k_{r+1}}(T,\rho_{r+1},\theta_{r+1}) 
\\ 
\vdotsS &\vdotsS & &\vdotsS \\
M_{k_{r+s}}(\ell,\rho_{r+s},\theta_{r+s}) &
M_{k_{r+s}}(\ell+1,\rho_{r+s},\theta_{r+s}) & \dots & 
M_{k_{r+s}}(T,\rho_{r+s},\theta_{r+s})
}\notag
\end{align}
where the sub-matrices $M(\cdot)$, $M_{k_i}(\cdot,\cdot)$, $i=1, \dots, r$, and $M_{k_{i}}(\cdot,\cdot,\cdot)$, $i=r+1, \dots, r+s$, are defined respectively in~\eqref{w=TLM}, \eqref{M-real} and \eqref{M-complex}, so $\mathbf{M}$ is known since $S$ is known.
Then, with $\mathbf{M}$ in~\eqref{boldM} and Claim~\ref{claim:exo_response}, we have that
\begin{align*}
W_0 =
\bmat{
       \tilde{\mathcal{T}} L(w(0)) M(\ell) & \dots & \tilde{\mathcal{T}} L(w(0)) M(T)}=
\tilde{\mathcal{T}} L(w(0)) \mathbf{M}
   =: L_0\mathbf{M}.
\end{align*}
So, we have found $L_0$ and $\mathbf{M}$ factorizing $W_0$, which proves Lemma~\ref{lemma:factorization_W0}.

\section{Proof of Lemma~\ref{lemma:factorization_W0_new}}
\label{app:lemma:factorization_W0_new}

From~\eqref{eq:W0},
\begin{align*}
W_0 = \bmat{ w(\ell) & S w(\ell) & \dots & S^{T-\ell} w(\ell) }.
\end{align*}
Since $\newM$ has full row rank, there exists $\mathsf{w} \in \R^{T-\ell+1}$ such that, for each $j \ge 1$,
\begin{align}
w(\ell) &  = \newM \mathsf{w} = \sum_{i=0}^{T-\ell} S^i w_\star \mathsf{w}_{i+1} = \left(\sum_{i=0}^{T-\ell} S^i \mathsf{w}_{i+1} \right) w_\star  =: \rho(S) w_\star \notag \\
S^j w(\ell) & = S^j \left(\sum_{i=0}^{T-\ell} S^i \mathsf{w}_{i+1} \right) w_\star =  \left(\sum_{i=0}^{T-\ell} S^i \mathsf{w}_{i+1} \right) S^j w_\star = \rho(S) S^j w_\star.
\end{align}
Hence,
\begin{align*}
W_0 & = \bmat{ \rho(S) w_\star & \rho(S) S w_\star & \dots & \rho(S) S^{T-\ell} w_\star } \\
& = \rho(S) \bmat{  w_\star & S w_\star & \dots & S^{T-\ell} w_\star } = \rho(S) \mathcal{M}
\end{align*}
and the claim is proven with $\mathcal{L} := \rho(S)$.

\section{Proof of Lemma~\ref{lemma:closed_loop_rep}}
\label{app:lemma:closed_loop_rep}

By $W_0 = \hat{L}_0 \hat{\mathbf{M}}$, \eqref{eq:data-matrices-identity} becomes $\Psi_1 = \bar{\Ab}  \Psi_0 +  \bar{\Bb} U_1 + \bar{\Pb} \Sb \hat{L}_0 \hat{\mathbf{M}}$.
Replace it in
\begin{align*}
& \Psi_1 \mathcal{G} = 
(\bar{\Ab}  \Psi_0 +  \bar{\Bb} U_1 + \bar{\Pb} \Sb \hat{L}_0 \hat{\mathbf{M}}) \mathcal{G} = \bmat{\bar{\Bb} & \bar{\Ab} & \bar{\Pb} \Sb \hat{L}_0 } \smat{U_1\\ \Psi_0\\ \hat{\mathbf{M}} } \mathcal{G} \overset{\eqref{eq:mathcal_G}}{=} \bar{\Ab} +  \bar{\Bb}\mathcal{K}.
\end{align*}

\section{Proof of Theorem~\ref{thm:outputregulation:aux}}
\label{app:thm:outputregulation:aux}

Let us show that items (i)-(ii) of the statement are verified.
To this end, set $\mathcal{G}:= \mathcal{Y} \mathcal{X}^{-1}$ by~\eqref{eq:sdp:x>0} so that, by~\eqref{eq:sdp:mathcal_Y} and \eqref{eq:K}, \eqref{eq:mathcal_G} holds and, by Lemma~\ref{lemma:closed_loop_rep}, the resulting closed-loop \eqref{eq:sys:xi_eta} is
\begin{equation}
\label{omega+,xi+,mu+}
\bmat{
\omega^+\\
\smat{
\xi^+\\
\mu^+
}
}
=
\bmat{
S				& 0\\
\bar{\Pb} \Sb	& \bar{\mathbf{A}}+\bar{\mathbf{B}} \mathcal{K} }
\bmat{
\omega\\
\smat{
\xi\\
\mu
}
}
=
\bmat{
S				& 0\\
\bar{\Pb} \Sb	& \Psi_1 \mathcal{G}}
\bmat{
\omega\\
\smat{
\xi\\
\mu
}
}.
\end{equation}
As for item (i) of the statement, \eqref{omega+,xi+,mu+} with $\omega = 0$ reduces to $\smat{\xi^+\\ \mu^+} = \Psi_1 \mathcal{G} \smat{\xi\\ \mu}$, which is globally asymptotically stable since $\Psi_1 \mathcal{G} = \Psi_1 \mathcal{Y} \mathcal{X}^{-1}$ is Schur by \eqref{eq:sdp:x>0} and \eqref{eq:sdp:schur}.
We then show item (ii) of the statement. Since the intersection of the spectra of $\Psi_1 \mathcal{G}$ and of $S$ is empty by Assumption~\ref{ass:S}, the Sylvester equation $(\Psi_1 \mathcal{G}) \Pi - \Pi S = - \bar{\Pb} \Sb$ has a solution $\Pi$ by \cite[Corollary A.1.1]{knobloch2012topics-in-control-theory}, i.e., there exists $\Pi =: \smat{\Pi_\xi\\ \Pi_\mu}$ such that
\begin{align}
\label{sylv.eq}
\bmat{
\Pi_\xi\\
\Pi_\mu }
S
=
\bar{\Pb} \Sb + \Psi_1 \mathcal{G} 
\bmat{
\Pi_\xi\\
\Pi_\mu }
\overset{\eqref{eq:sys:xi_eta:barAbarB},\eqref{omega+,xi+,mu+}}{=} \bmat{\Lb Z_{2} \\ GZ_{2}} \Sb + (\bar{\Ab} +  \bar{\Bb}\mathcal{K}) \bmat{
\Pi_\xi\\
\Pi_\mu }
.
\end{align}
With the partition $\mathcal{K} = \smat{\mathcal{K}_\xi & \mathcal{K}_\mu}$, $\mathcal{K}_\xi \in \mathbb{R}^{m \times (m+p)\ell}$ and $\mathcal{K}_\mu \in \mathbb{R}^{m \times pd}$, write $\bar{\Ab} +  \bar{\Bb}\mathcal{K} \overset{\eqref{eq:sys:xi_eta:barAbarB}}{=} \smat{
\Ab + \Bb \mathcal{K}_\xi & \Bb \mathcal{K}_\mu\\
G Z_1 & \Phi } $ so that the second block row in~\eqref{sylv.eq} is
\begin{align}
\label{sylv.eq.2nd.row}
\Pi_\mu S = G( Z_2 \Sb + Z_1 \Pi_\xi ) + \Phi \Pi_\mu.
\end{align}
It can be shown that, for $\Phi$ and $G$ defined as in~\eqref{eq:im-matrices}, \eqref{sylv.eq.2nd.row} implies
\begin{align}
\label{sylv.eq.2nd.row.implication}
Z_2 \Sb + Z_1 \Pi_\xi = 0
\end{align}
by following the same steps as in \cite[Lemma 4.3 and its proof]{isidori2017lectures}.
Indeed, let $\Pi_\mu$ be partitioned consistently with the partition of $\Phi$ as
\begin{align*}
\Pi_\mu =: \smat{ \Pi_{\mu,1} \\ \vdotsS \\ \Pi_{\mu,d}}.
\end{align*}
By~\eqref{eq:im-matrices}, \eqref{sylv.eq.2nd.row} rewrites as
\begin{align*}
\smat{ \Pi_{\mu,1} S \\ \vdotsS \\ \Pi_{\mu,d-1} S \\ \Pi_{\mu,d} S } = 
\smat{\Pi_{\mu,2} \\ \vdotsS \\ \Pi_{\mu,d} \\ -s_0 \Pi_{\mu,1} - s_1 \Pi_{\mu,2} - \dots - s_{d-1} \Pi_{\mu,d} + Z_2 \Sb + Z_1 \Pi_\xi}.
\end{align*}
The first $d-1$ block rows yield, for $i = 1, \dots, d$, $\Pi_{\mu,i} = \Pi_{\mu,1} S^{i-1}$; by substituting these relations, the last block row rewrites
\begin{align*}
& \Pi_{\mu,d} S = \Pi_{\mu,1} S^d = - s_0 \Pi_{\mu,1} - s_1 \Pi_{\mu,1} S - \dots - s_{d-1} \Pi_{\mu,1} S^{d-1} + Z_2 \Sb + Z_1 \Pi_\xi \\
& \iff Z_2 \Sb + Z_1 \Pi_\xi  = \Pi_{\mu,1} (S^d + s_{d-1} S^{d-1} + \dots + s_1 S + s_0 I) = 0  
\end{align*}
since the minimal polynomial of $S$ annihilates $S$.

With classical arguments from output regulation as in, e.g., \cite[p.~13-14]{knobloch2012topics-in-control-theory}, we consider the global and invertible change of variables $(\omega,\tilde{\xi},\tilde{\mu}):= (\omega,\xi - \Pi_\xi \omega,\mu - \Pi_\mu \omega)$ in~\eqref{omega+,xi+,mu+} and in the output equation $\varphi = Z_1 \xi + Z_2 \Sb \omega$ from~\eqref{eq:sys:xi_eta:phi}; this change of variables rewrites them as
\begingroup
\thinmuskip=1mu plus 1mu minus 1mu
\medmuskip=1mu plus 2mu minus 2mu
\thickmuskip=1mu plus 3mu minus 3mu
\begin{align*}
\omega^+ & = S \omega\\
\smat{\tilde{\xi}^+ \\ \tilde{\mu}^+} &  \overset{\eqref{omega+,xi+,mu+}}{=}  (\bar{\Ab} + \bar{\Bb} \mathcal{K}) \smat{\tilde{\xi}\\ \tilde{\mu}} + \left( (\bar{\Ab} + \bar{\Bb} \mathcal{K}) \smat{\Pi_\xi\\ \Pi_\mu} + \bar{\Pb} \Sb - \smat{\Pi_\xi\\ \Pi_\mu} S \right) \omega \overset{\eqref{sylv.eq}}{=}  (\bar{\Ab} + \bar{\Bb} \mathcal{K}) \smat{\tilde{\xi}\\ \tilde{\mu}} \\
\varphi & = Z_1 \tilde{\xi} + (Z_1 \Pi_\xi + Z_2 \Sb) \omega \overset{\eqref{sylv.eq.2nd.row.implication}}{=} Z_1 \tilde{\xi}.
\end{align*}
\endgroup
Since $\bar{\Ab} +  \bar{\Bb}\mathcal{K}$ was shown to be Schur,  we have that for each initial condition $(\omega(0),\tilde{\xi}(0),\tilde{\mu}(0)):= (\omega(0),\xi(0) - \Pi_\xi \omega(0),\mu(0) - \Pi_\mu \omega(0))$, $\lim_{k \to \infty} \smat{\tilde{\xi}(k)\\ \tilde{\mu}(k)} = \smat{0\\ 0}$ and $\lim_{k \to \infty} \varphi(k) = 0$, so that item (ii) of the statement is verified.

\section{Proof of Proposition~\ref{proposition:ex_remark}}
\label{app:proposition:ex_remark}

The factorization $W_0 = \hat{L}_0 \hat{\mathbf{M}}$, the matrices $\Obs$, $\Tu$, $\Tw$ in~\eqref{eq:pre:matrix:OTR:O}-\eqref{eq:pre:matrix:OTR:T_P}, and the definition $\mathcal{S}_S := \smat{S^{-\ell} \\ \vdotsS \\ S^{-1}}$ yield, see \eqref{eq:claim1:yellxy:yell}, that
\begingroup
\thinmuskip=1mu plus 1mu minus 1mu
\medmuskip=2mu plus 2mu minus 2mu
\thickmuskip=3mu plus 3mu minus 3mu
\begin{align*}
\underbrace{
\left[
\begin{smallarray}{cccc}
    y(0) & y(1) & \dotsS & y(T-\ell )\\
    \vdotsS & \vdotsS &   & \vdotsS \\
    y(\ell -1) & y(\ell ) & \dotsS & y(T-1) \\
    \hline
    u(0) & u(1) & \dotsS & u(T-\ell) \\
    \vdotsS & \vdotsS &  & \vdotsS\\
    u(\ell -1) & u(\ell) & \dotsS & u(T-1)\\
    \hline
    \eta(\ell) & \eta(\ell + 1) & \dotsS & \eta(T)\\
    \hline
    \multicolumn{4}{c}{\hat{\mathbf{M}}}
\end{smallarray}
\right]}_{=\smat{
    \Psi_0\\
    \hat{\mathbf{M}}
}} &=
\underbrace{
\left[
\begin{smallarray}{cccc}
    \Obs & \Tu & \Tw \mathcal{S}_S \hat{L}_0 & 0 \\
    \hline
    0 & I_{m \ell} & 0 & 0 \\
    \hline
    0 & 0 & 0 & I_{pd} \\
    \hline
    0 & 0 & I_{\hat{n}_w} & 0
\end{smallarray}
\right]
}_{=:\mathcal{N}_1}
\underbrace{
\left[
\begin{smallarray}{cccc}
    x(0) & x(1) & \dotsS & x(T-\ell) \\
    \hline
    u(0) & u(1) & \dotsS & u(T-\ell) \\
    \vdotsS & \vdotsS &  & \vdotsS \\
    u(\ell-1) & u(\ell) & \dotsS & u(T-1) \\
    \hline
    \multicolumn{4}{c}{\hat{\mathbf{M}}} \\  
    \hline
    \eta(\ell) & \eta(\ell + 1) & \dotsS & \eta(T)
\end{smallarray}
\right]}_{=:\mathcal{N}_2}
\end{align*}
\endgroup
with
$\mathcal{N}_1 \in \R^{ p\ell + m\ell + pd + \hat{n}_w \times n + m\ell + \hat{n}_w + pd}$ and $\mathcal{N}_2 \in \R^{n + m\ell + \hat{n}_w + pd \times T - \ell + 1}$.
By the hypothesis $p \ell > n$, we have
\begin{align}
\label{conseq_pl_>_n}
p\ell + m\ell + \hat{n}_w + pd>
n + m\ell + \hat{n}_w + pd
\end{align}
where the left hand side is the number of rows of $\smat{\Psi_0\\ \hat{\mathbf{M}}}$ and the right hand side is the number of columns of $\mathcal{N}_1$.
By the factorization $\smat{\Psi_0\\ \hat{\mathbf{M}}} = \mathcal{N}_1 \mathcal{N}_2$ and rank inequalities \cite[\S 0.4.5]{horn2013matrix}, $\rank \smat{\Psi_0\\ \hat{\mathbf{M}}} \le \rank \mathcal{N}_1 \le n + m\ell + \hat{n}_w + pd$ so that, by~\eqref{conseq_pl_>_n}, $\smat{\Psi_0\\ \hat{\mathbf{M}}}$ is \emph{not} full row rank.
By the hypothesis of $\hat{\mathbf{M}}$ full row rank, there is a \emph{strict}, possibly empty, subset of the rows of $\Psi_0$, called $\Psi_{0,\rm row}$, such that $\smat{\Psi_{0,\rm row}\\ \hat{\mathbf{M}}}$ has full row rank.
Let $\Psi_{1,\rm row}$ be the remaining, possibly all, rows of $\Psi_0$ so that there is a permutation matrix $\mathcal{Q}$ satisfying
$
\Psi_0 = \mathcal{Q} \smat{\Psi_{0,\rm row}\\ \Psi_{1,\rm row} }$.
Since the rows of $\Psi_0$ contained in $\Psi_{1,\rm row}$ are linearly dependent on the rows of $\Psi_{0,\rm row}$ and $\hat{\mathbf{M}}$, there are some matrices $\mathcal{R}_0$ (possibly with no columns) and $\mathcal{R}_{\hat{\mathbf{M}}}$ such that
\begin{align*}
\Psi_{1,\rm row} & = \mathcal{R}_0 \Psi_{0,\rm row} + \mathcal{R}_{\hat{\mathbf{M}}} \hat{\mathbf{M}} \\
\Psi_0 & = \mathcal{Q} \smat{\Psi_{0,\rm row}\\ \Psi_{1,\rm row} } = \mathcal{Q} \smat{I & 0\\ \mathcal{R}_0 & \mathcal{R}_{\hat{\mathbf{M}}} } \smat{\Psi_{0,\rm row}\\ \hat{\mathbf{M}} } =  
\mathcal{Q} \smat{I \\ \mathcal{R}_0 } \Psi_{0,\rm row} +
\mathcal{Q} \smat{ 0\\ \mathcal{R}_{\hat{\mathbf{M}}} } \hat{\mathbf{M}} \\
& =: \mathcal{S}_{\Psi_0} \Psi_{0,\rm row} + \mathcal{S}_{\hat{\mathbf{M}}} \hat{\mathbf{M}}
\end{align*}
where $\mathcal{S}_{\Psi_0}$ is full column rank (possibly with no columns) since $\mathcal{Q}$ is invertible and $\smat{I \\ \mathcal{R}_0 }$ is full column rank (possibly with no columns).
Then, if \eqref{eq:sdp} is feasible for some $\mathcal{X}$ and $\mathcal{Y}$, \eqref{eq:sdp:x>0} and \eqref{eq:sdp:mathcal_Y} yield $\mathcal{X} \succ 0$, $I_{(m+p)\ell+pd} = \Psi_0 \mathcal{Y}  \mathcal{X}^{-1}$, $0 = \hat{\mathbf{M}} \mathcal{Y} \mathcal{X}^{-1}$ and, thus, $I_{(m+p)\ell+pd} = (\mathcal{S}_{\Psi_0} \Psi_{0,\rm row} + \mathcal{S}_{\hat{\mathbf{M}}} \hat{\mathbf{M}} ) \mathcal{Y}  \mathcal{X}^{-1} = \mathcal{S}_{\Psi_0} \Psi_{0,\rm row} \mathcal{Y}  \mathcal{X}^{-1}$ and $\rank I_{(m+p)\ell+pd} \le  \rank \mathcal{S}_{\Psi_0}$.
Since $\mathcal{S}_{\Psi_0}$ has more rows than columns and is full column rank, the last inequality is a contradiction and we conclude that \eqref{eq:sdp} is infeasible.

\section{Proof of Lemma~\ref{lem:rel_sol_sys+exo_auxsys+auxexo}}
\label{app:lem:rel_sol_sys+exo_auxsys+auxexo}

    For $(A,C)$ observable, we need to show that for each $(\hat{w},\hat{x},\hat{\chi})$ and sequence $\{u(k)\}_{k = 0}^{\infty}$,
    \begin{itemize}[leftmargin=*]
        \item [(i)] 
        the solution $(w(\cdot),x(\cdot),\chi(\cdot))$ to \eqref{eq:pre:sys_cl_xy} with initial condition $(w(0),x(0),\chi(0))\!=\!(\hat{w},\hat{x},\hat{\chi})$ and with input $\{u(k)\}_{k = 0}^{\infty}$, and the corresponding output response $y(\cdot) = C x(\cdot) + Q w(\cdot)$ satisfy, for all $k \ge \ell$,
    \begin{equation*}
        \chi(k) = (y(k - \ell),\dots,y(k-1),u(k - \ell),\dots,u(k - 1));
    \end{equation*}
    \item [(ii)]
    there exists $\hat{\xi}$ such that: such solution $\big(w(\cdot),x(\cdot),\chi(\cdot)\big)$, the corresponding output response $y(\cdot) = C x(\cdot) + Q w(\cdot)$, the solution $(\omega(\cdot),\xi(\cdot))$ to \eqref{eq:aux_all_except_control} with initial condition $(\omega(\ell), \xi(\ell)) = (S^\ell \hat{w} , \hat{\xi})$ and with input $
    \{v(k)\}_{k = \ell}^{\infty} = \{u(k)\}_{k = \ell}^{\infty}$, the corresponding output response $\varphi(\cdot) = Z_1 \xi(\cdot) + Z_2 \Sb \omega(\cdot)$ satisfy, for all $k \ge \ell$,%
\begin{subequations}
       \begin{align}
        &\xi(k) = (y(k - \ell),\dots,y(k - 1),u(k - \ell),\dots,u(k - 1)), \label{eq:lem3:xi_y:xi}\\
         &\varphi(k) = y(k) \label{eq:lem3:xi_y:y}.
    \end{align}
   \end{subequations}
    \end{itemize}
    
As for (i), we use mathematical induction.
As for the base case, we need to show that $\chi(\ell) = \big( y(0), \dots, y(\ell-1), u(0), \dots, u(\ell-1) \big)$.
Indeed, by~\eqref{eq:pre:sys_cl_xy},
\begin{align*}
            \chi(1) = \left[
            \begin{smallarray}{c}
               \hat{\chi}_2 \\
                \vdotsS\\
                \hat{\chi}_\ell\\
                y(0)\\
                \hline
                  \hat{\chi}_{\ell + 2} \\
                \vdotsS\\
                \hat{\chi}_{\ell + \ell}\\
                u(0)     
            \end{smallarray}
            \right],\dots,
            \chi(\ell - 1) = \left[
            \begin{smallarray}{c}
                   \hat{\chi}_\ell \\
               y(0)\\
                \vdotsS\\
                y(\ell - 2)\\
                \hline
                   \hat{\chi}_{\ell + \ell} \\
                u(0)\\
                \vdotsS\\
                u(\ell - 2) 
            \end{smallarray}
            \right],
            \chi(\ell) = \left[
            \begin{smallarray}{c}
               y(0)\\
                \vdotsS\\
                y(\ell - 1)\\
                \hline
                u(0)\\
                \vdotsS\\
                u(\ell - 1)         
            \end{smallarray}
            \right].
\end{align*}
As for the induction step, we suppose that, for some $k \ge \ell$,
        \begin{equation}
        \label{eq:induction_step}
            \chi(k) = \left[
            \begin{smallarray}{c}
               y(k - \ell)\\
                \vdotsS\\
                y(k - 1)\\
                \hline
                u(k - \ell)\\
                \vdotsS\\
                u(k - 1)        
            \end{smallarray}
            \right],
        \end{equation}
and we need to show that $
            \chi(k+1) = \left[
            \begin{smallarray}{c}
               y(k - \ell + 1)\\
                \vdotsS\\
                y(k)\\
                \hline
                u(k - \ell + 1)\\
                \vdotsS\\
                u(k)       
            \end{smallarray}
            \right]$.
Indeed,
        \begin{align*}
            \chi(k + 1) & \overset{\eqref{eq:pre:sys_cl_xy:chi}}{=} \Fb \chi(k) + \Lb y(k) + \Bb u(k)  =  \left[
            \begin{smallarray}{c}
               \chi_2(k)\\
                \vdotsS\\
                \chi_\ell(k)\\
                y(k)\\
                \hline
                \chi_{\ell + 2}(k)\\
                \vdotsS\\
                \chi_{\ell + \ell}(k)\\
                u(k)        
            \end{smallarray}
            \right] 
            \overset{\eqref{eq:induction_step}}{=}
            \left[
            \begin{smallarray}{c}
               y(k - \ell + 1)\\
                \vdotsS\\
                y(k)\\
                \hline
                u(k - \ell + 1)\\
                \vdotsS\\
                u(k)        
            \end{smallarray}
            \right].
        \end{align*}

As for (ii), we start noting some properties of the solutions to~\eqref{eq:pre:sys_cl_xy} and \eqref{eq:aux_all_except_control}.
For each $\hat{w}$, the solution $w(\cdot)$ to~\eqref{eq:pre:sys_cl_xy:w} with initial condition $w(0)= \hat{w}$ satisfies, for each $k \ge 0$,
\begin{align}
\label{eq:lem3:sol_exosys}
w(k) = S^k \hat{w}.
\end{align}
Moreover, we can apply Claim~\ref{claim:1:xy} to the components $(w(\cdot),x(\cdot))$ of the solution $\big(w(\cdot),x(\cdot),\chi(\cdot)\big)$ to~\eqref{eq:pre:sys_cl_xy} and obtain that for each $(\hat{w},\hat{x})$ and sequence $\{u(k)\}_{k = 0}^{\infty}$, the solution $(w(\cdot), x(\cdot))$ to~\eqref{eq:pre:sys_cl_xy} with initial condition $(w(0), x(0))\!=\!(\hat{w},\hat{x})$ and input $\{u(k)\}_{k = 0}^{\infty}$ satisfies, for each $k \ge \ell$
    \begin{equation}\label{eq:lem3:y}
        \left[
            \begin{smallmatrix}
                y(k - \ell)\\
                \vdotsS\\
                y(k - 1)
            \end{smallmatrix}
            \right] = \Obs x(k - \ell) + \Tu \left[
            \begin{smallmatrix}
                u(k - \ell)\\
                \vdotsS\\
                u(k - 1)
            \end{smallmatrix}
            \right] + \Tw \left[
            \begin{smallmatrix}
                w(k - \ell)\\
                \vdotsS\\
                w(k - 1)
            \end{smallmatrix}
            \right].
    \end{equation} 
Finally, the solution $\omega(\cdot)$ to~\eqref{eq:aux_all_except_control:omega} with initial condition $\omega(\ell)= S^\ell \hat{w}$ satisfies for each $k \ge \ell$
\begin{align}
\label{eq:lem3:sol_omega}
\Sb \omega(k) & \overset{\eqref{eq:pre:matrix_SABZ:S}}{=} \smat{S^{-\ell} \\ \vdotsS \\ S^{-1} \\ I_{n_w} } \omega(k) \overset{\eqref{eq:aux_all_except_control:omega}}{=}  \smat{S^{-\ell} \\ \vdotsS \\ S^{-1} \\ I_{n_w} } S^{k-\ell} S^\ell \hat{w}
= \smat{S^{k-\ell} \\ \vdotsS \\ S^{k-1} \\ S^k } \hat{w}   \overset{\eqref{eq:lem3:sol_exosys}}{=} \smat{w(k-\ell) \\ \vdotsS \\ w(k-1) \\ w(k) }.
\end{align}
With these properties in mind, we can prove (ii). To this end, for each $(\hat{w},\hat{x})$ and $\{u(k)\}_{k = 0}^{\infty}$, select $\hat{\xi} = (\hat{\xi}_1,\dots,\hat{\xi}_{\ell},\hat{\xi}_{\ell + 1},\dots,\hat{\xi}_{\ell + \ell})$ such that 
\begin{align}\label{eq:lem3:hat_xi_ini}
\smat{
\hat{\xi}_1\\
\vdotsS\\
\hat{\xi}_\ell
}
= \Obs \hat{x} + \Tu 
\smat{
u(0)\\
\vdotsS\\
u(\ell - 1)
} + \Tw
\smat{
I \\
\vdotsS\\
S^{\ell - 1}
}\hat{w}
\text{ and }
\smat{
\hat{\xi}_{\ell + 1}\\
\vdotsS\\
\hat{\xi}_{\ell + \ell} 
}
=
\smat{
u(0)\\
\vdotsS\\
u(\ell - 1)
}.
\end{align}
Let us prove \eqref{eq:lem3:xi_y:xi} by mathematical induction.
As for the base case, we need to show that $\xi(\ell) = (y(0),\dots,y(\ell - 1),u(0),\dots,u(\ell - 1))$.
Indeed,
\begin{align*}
\xi(\ell) & = \hat{\xi} 
\overset{\eqref{eq:lem3:hat_xi_ini}}{=} 
\left[
\begin{array}{c}
\Obs\hat{x} + \Tu \smat{ u(0)\\ \vdotsS\\ u(\ell - 1) } + \Tw \smat{
I \\
\vdotsS\\
S^{\ell - 1}
}\hat{w}\\
\hline
\smat{ u(0)\\ \vdotsS\\ u(\ell - 1) }
\end{array}
\right]\\
& \overset{\eqref{eq:lem3:sol_exosys}}{=} 
\left[
\begin{array}{c}
\Obs\hat{x} + \Tu \smat{ u(0)\\ \vdotsS\\ u(\ell - 1) } + \Tw \smat{w(0)\\ \vdotsS \\ w(\ell-1)} \\
\hline
\smat{ u(0)\\ \vdotsS\\ u(\ell - 1) }
\end{array}
\right]
        \overset{\eqref{eq:lem3:y}}{=} \left[
        \begin{smallarray}{c}
               y(0)\\
                \vdotsS\\
                y(\ell - 1)\\
            \hline   
                u(0)\\
                \vdotsS\\
                u(\ell - 1)
        \end{smallarray}
        \right]. 
    \end{align*}
As for the induction step, we suppose that for $k \ge \ell$
    \begin{align}\label{eq:lem3:xik=yu}
        \xi(k) = \left[
        \begin{smallarray}{c}
        \xi_1(k)\\
        \vdotsS\\
        \xi_\ell(k)\\
        \hline
        \xi_{\ell + 1}(k)\\
        \vdotsS\\
        \xi_{\ell + \ell}(k)
        \end{smallarray}
        \right] = \left[
        \begin{smallarray}{c}
        y(k-\ell)\\
        \vdotsS\\
        y(k-1)\\
        \hline
        u(k-\ell)\\
        \vdotsS\\
        u(k-1)
        \end{smallarray}
        \right]
    \end{align}
    and we need to show that $\xi(k + 1) = (y(k - \ell + 1),\dots,y(k),u(k - \ell + 1),\dots, u(k))$.
From~\eqref{eq:aux_all_except_control:xi} we have that
\begin{align*}
&         \xi(k + 1) = \left[
        \begin{smallarray}{c}
             \xi_1(k+1)\\
                \vdotsS\\
                \xi_{\ell -1}(k+1)\\
                \xi_\ell(k+1)\\
            \hline
                \xi_{\ell + 1}(k+1)\\
                \vdotsS\\
                \xi_{\ell + \ell - 1}(k+1)\\
                \xi_{\ell + \ell}(k+1)
        \end{smallarray}
        \right]
        = \Ab \xi(k) + \Bb v(k) + \Lb Z_2 \Sb \omega(k) \! = \! \left[
        \begin{smallarray}{c}
                \xi_2(k)\\
                \vdotsS\\
                \xi_{\ell}(k)\\
                Z_1 \xi(k) + Z_2 \Sb \omega(k)\\
            \hline
                \xi_{\ell + 2}(k)\\
                \vdotsS\\
                \xi_{\ell + \ell}(k)\\
                v(k)
        \end{smallarray}
        \right]\\
        & \overset{\eqref{eq:lem3:xik=yu},\eqref{eq:lem3:sol_omega}}{=} \left[
        \begin{smallarray}{c}
               y(k - \ell + 1)\\
                \vdotsS\\
                y(k - 1)\\
                Z_1 \left(y(k - \ell),\dots,y(k - 1),u(k - \ell),\dots,u(k - 1)\right) + Z_2 \left(w(k - \ell),\dots,w(k) \right)\\
            \hline  
                u(k - \ell +1)\\
                \vdotsS\\
                u(k-1)\\
                v(k)
        \end{smallarray}
        \right]\\
        & \overset{\eqref{eq:pre:matrix_SABZ:Z}}{=} \left[
        \begin{smallarray}{c}
               y(k - \ell + 1)\\
                \vdotsS\\
                y(k - 1)\\
                CA^\ell \Obs^{\tu{L}}
                \smat{
                    y(k - \ell)\\
                    \vdotsS\\
                    y(k - 1)}
                + (C \Ru - CA^\ell \Obs^{\tu{L}
                }\Tu)
                \smat{
                    u(k - \ell)\\
                    \vdotsS\\
                    u(k - 1)} + \smat{
C\Rw - CA^{\ell} \Obs^{\tu{L}}\Tw & Q}  \smat{
                    w(k - \ell)\\
                    \vdotsS\\
                    w(k-1)\\
                    w(k)}\\
            \hline 
                u(k - \ell +1)\\
                \vdotsS\\
                u(k-1)\\
                u(k)
        \end{smallarray}
        \right]
\end{align*}
    where the last equality follows from the statement since $v(j) = u(j)$ for all $j \ge \ell$.
From the last equation, the induction step is proven if we show that 
\begin{align}\label{eq:lem2:yk}
CA^{\ell} \Obs^{\tu{L}}
\smat{
y(k-\ell)\\
\vdotsS\\
y(k-1)} 
+ (C\Ru - CA^{\ell} \Obs^{\tu{L}}\Tu)
\smat{
u(k-\ell)\\
\vdotsS \\
u(k-1)
} \\
+ \smat{
C\Rw - CA^{\ell} \Obs^{\tu{L}}\Tw & Q}
\smat{
w(k-\ell)\\
\vdotsS \\
w(k-1)\\
w(k)} = y(k). \notag
    \end{align}
This holds by~\eqref{eq:claim1:yellxy:y} in Claim~\ref{claim:1:xy} and proves \eqref{eq:lem3:xi_y:xi}.  
    Let us prove \eqref{eq:lem3:xi_y:y}.
The solution $y$ satisfies \eqref{eq:claim1:yellxy:y} in Claim~\ref{claim:1:xy} for each $k \ge \ell$ and, by~\eqref{eq:pre:matrix_SABZ:Z}, for each $k \ge \ell$,
\begin{align*}
y(k) & = Z_1 \big(y(k - \ell),\dots,y(k - 1),u(k - \ell),\dots,u(k - 1)\big) \\
&  + Z_2 \big(w(k - \ell),\dots,w(k) \big)  \overset{\eqref{eq:lem3:xi_y:xi},\eqref{eq:lem3:sol_omega}}{ =} Z_1 \xi(k) + Z_2 \Sb \omega(k)\overset{\eqref{eq:aux_all_except_control:varphi}}{=} \varphi(k).
\end{align*}
Since we have shown \eqref{eq:lem3:xi_y:xi} and \eqref{eq:lem3:xi_y:y}, (ii) holds.

\section{Proof of Theorem~\ref{thm:outreg_sys=outreg_auxsys}}
\label{app:thm:outreg_sys=outreg_auxsys}

Consider the closed-loop \eqref{eq:sys:aux}, \eqref{eq:ctrl:aux} given by
\begin{subequations}\label{eq:aux_all_cl}
    \begin{align}
\omega^+ &= S \omega \label{eq:aux_all_cl:omega}\\
\xi^+ & = \Ab \xi + \Bb v  + \Lb Z_{2} \Sb \omega \label{eq:aux_all_cl:xi}\\
\varphi &= Z_1 \xi + Z_2 \Sb \omega \label{eq:aux_all_cl:varphi}\\
 \mu^+ &= \Phi \mu + G \varphi\label{eq:aux_all_cl:eta}\\
 v & = \mathcal{K}\smat{\xi\\ \mu} .\label{eq:aux_all_cl:v}
\end{align}
\end{subequations}
For each $(\hat{w}, \hat{x}, \hat{\chi}, \hat{\eta})$, the solution $(w(\cdot),x(\cdot),\chi(\cdot),\eta(\cdot))$ to~\eqref{eq:sys_cl_xy2}
with initial condition $(w(0),x(0),\chi(0),\eta(0)) = (\hat{w},\hat{x},\hat{\chi},\hat{\eta})$ has $\{u(k)\}_{k = 0}^{\infty} = \Big\{\mathcal{K} \smat{\chi(k)\\ \eta(k)} \Big\}_{k = 0}^{\infty}$.
Lemma~\ref{lem:rel_sol_sys+exo_auxsys+auxexo} yields that for such $(\hat{w}, \hat{x}, \hat{\chi})$ and $\{u(k)\}_{k = 0}^{\infty}= \Big\{\mathcal{K} \smat{\chi(k)\\ \eta(k)} \Big\}_{k = 0}^{\infty}$, there exists $\hat{\xi}$ as in \eqref{eq:lem3:hat_xi_ini} such that:
\begin{itemize}[leftmargin=*]
\item the solution $(w(\cdot),x(\cdot),\chi(\cdot))$ to~\eqref{eq:sys_cl_xy2:w}-\eqref{eq:sys_cl_xy2:chi} with initial condition equal to $(w(0),x(0),\chi(0)) = (\hat{w},\hat{x},\hat{\chi})$ and input $\{u(k)\}_{k = 0}^{\infty} = \Big\{\mathcal{K} \smat{\chi(k)\\ \eta(k)} \Big\}_{k = 0}^{\infty}$ and 
\item the corresponding output response $y(\cdot) = C x(\cdot) + Q w(\cdot)$,
\item the solution $(\omega(\cdot), \xi(\cdot))$ to~\eqref{eq:aux_all_cl:omega}-\eqref{eq:aux_all_cl:varphi} with initial condition $(\omega(\ell), \xi(\ell)) = (S^\ell \hat{w}, \hat{\xi})$ and input $\{v(k)\}_{k = \ell}^{\infty} = \{u(k)\}_{k = \ell}^{\infty}$ and
\item the corresponding output response $\varphi(\cdot) = Z_1 \xi(\cdot) + Z_2 \Sb \omega(\cdot)$
\end{itemize}
satisfy for each $k \ge \ell$
\begin{equation}
\label{eq:lem8:xi_beta_alpha}
\xi(k) = \left[\begin{smallarray}{c}
        y(k - \ell)\\
        \vdotsS\\
        y(k - 1)\\
        \hline
        u(k - \ell)\\
        \vdotsS\\
        u(k - 1)
        \end{smallarray}
        \right] = \chi(k), \quad \varphi(k) = y(k) .
\end{equation}
Moreover, for each $(\hat{w},\hat{x},\hat{\chi},\hat{\eta})$ there exists $\hat{\mu}$,
in addition to the $\hat{\xi}$ above, such that the above solution $(w(\cdot),x(\cdot),\chi(\cdot))$ to \eqref{eq:sys_cl_xy2:w}-\eqref{eq:sys_cl_xy2:chi} and the corresponding output response $y(\cdot)$, the above solution $(\omega(\cdot),\xi(\cdot))$ to~\eqref{eq:aux_all_cl:omega}-\eqref{eq:aux_all_cl:varphi} and the corresponding output response $\varphi(\cdot)$, the solution $\eta(\cdot)$ to~\eqref{eq:sys_cl_xy2:eta} with initial condition $\eta(0) = \hat{\eta}$ and fictitious input $\{y(k)\}_{k=0}^\infty$, 
and the solution $\mu(\cdot)$ to~\eqref{eq:aux_all_cl:eta} with initial condition $\mu(\ell) = \hat{\mu}$ and fictitious input $\{\varphi(k)\}_{k=\ell}^\infty=\{y(k)\}_{k=\ell}^\infty$, by~\eqref{eq:lem8:xi_beta_alpha}, also satisfy
\begin{align}
\label{eq:eta=mu}
\eta(k) & = \mu(k) \qquad \forall k \ge \ell.
\end{align}
Indeed, for each $\hat{\eta}$, take
\begin{align}
\label{eq:hatmu}
\hat{\mu} = \eta(\ell) = \Phi^\ell \hat{\eta} +  \sum_{i=0}^{\ell-1} \Phi^{\ell-i-1} G y(i) .
\end{align}
Then, $\eta(\ell) = \mu(\ell)$ (same initial condition) and $\{y(k)\}_{k=\ell}^\infty=\{\varphi(k)\}_{k=\ell}^\infty$ (same ``input'') ensure \eqref{eq:eta=mu}.

Let us summarize what we have obtained so far.
We have considered an arbitrary initial condition $(w(0),x(0),\chi(0),\eta(0)) = (\hat{w},\hat{x},\hat{\chi},\hat{\eta})$ and the input $\{u(k)\}_{k = 0}^{\infty} = \Big\{\mathcal{K} \smat{\chi(k)\\ \eta(k)} \Big\}_{k = 0}^{\infty}$ for the open loop \eqref{eq:sys_cl_xy2:w}-\eqref{eq:sys_cl_xy2:eta} and obtained a solution $(w(\cdot),x(\cdot),\chi(\cdot),\eta(\cdot))$; with this solution, we have associated a solution $(\omega(\cdot),\xi(\cdot),\mu(\cdot))$ to the open loop \eqref{eq:aux_all_cl:omega}-\eqref{eq:aux_all_cl:eta} with some initial condition $(\omega(\ell),\xi(\ell),\mu(\ell)) = ( S^\ell \hat{w},\hat{\xi},\hat{\mu})$, with $\hat{\xi}$ in~\eqref{eq:lem3:hat_xi_ini} and $\hat{\mu}$ in~\eqref{eq:hatmu}, and input $\{ v(k) \}_{k=\ell}^{\infty} = \Big\{\mathcal{K} \smat{\chi(k)\\ \eta(k)} \Big\}_{k = \ell}^{\infty}$ such that $\xi(k) = \chi(k)$ and $\mu(k) = \eta(k)$ for all $k \ge \ell$, see \eqref{eq:lem8:xi_beta_alpha}-\eqref{eq:eta=mu}.
As a consequence, $\{ v(k) \}_{k=\ell}^{\infty} = \Big\{\mathcal{K} \smat{\chi(k)\\ \eta(k)} \Big\}_{k = \ell}^{\infty} = \Big\{\mathcal{K} \smat{\xi(k)\\ \mu(k)} \Big\}_{k = \ell}^{\infty} $.
So, for some initial condition $(\omega(\ell), \xi(\ell), \mu(\ell))$, $(\omega(\cdot), \xi(\cdot), \mu(\cdot) )$ is also a solution to the closed-loop \eqref{eq:aux_all_cl}.

Based on these relations, we can show that the control law in~\eqref{eq:sys_cl_xy2:chi}-\eqref{eq:sys_cl_xy2:u} verifies the statement.
As for item~(i) of the statement, if $w(0) = \hat{w} = 0$ and $w(k) = 0$ for $k \ge 0$, we have that $\omega(\ell) =  S^\ell \hat{w} = 0 $ and $\omega(k) = 0$ for all $k \ge \ell$.
By Theorem~\ref{thm:outputregulation:aux}, $\xi(\cdot)$ and $\mu(\cdot)$ converge asymptotically to $0$.
By~\eqref{eq:lem8:xi_beta_alpha}, if $\xi(\cdot)$ converges asymptotically to $0$, so do $k \mapsto \smat{y(k - \ell)\\ \vdotsS\\ y(k - 1)}$, $k \mapsto \smat{u(k - \ell)\\ \vdotsS\\ u(k - 1)}$ and $k \mapsto \chi(k)$.
By $(A,C)$ observable and \eqref{eq:claim1:yellxy:x} in Claim \ref{claim:1:xy}, we have that also $k \mapsto x(k)$ converges asymptotically to $0$.
By~\eqref{eq:eta=mu}, if $\mu(\cdot)$ converges asymptotically to $0$, so does $k \mapsto \eta(k)$.
In summary, we have shown that for each $(\hat{w},\hat{x},\hat{\chi},\hat{\eta})$ with $\hat{w}=0$, the solution $(w(\cdot),x(\cdot),\chi(\cdot),\eta(\cdot))$ to~\eqref{eq:sys_cl_xy2} with initial condition $(0,\hat{x},\hat{\chi},\hat{\eta})$ satisfies $\lim_{k \to \infty} x(k) = 0$, $\lim_{k \to \infty} \chi(k) = 0$ and $\lim_{k \to \infty} \eta(k) = 0$.
For linear time-invariant systems, attractivity implies stability, see, e.g., \cite[Prop.~4]{lewis2017remarks}, and global attractivity.
As for item~(ii) of the statement, by Theorem~\ref{thm:outputregulation:aux}, $\varphi$ converges asymptotically to $0$ for each initial condition $(\omega(\ell),\xi(\ell),\mu(\ell))= ( S^\ell \hat{w},\hat{\xi},\hat{\mu})$.
By~\eqref{eq:lem8:xi_beta_alpha}, if $\varphi(\cdot)$ converges asymptotically to $0$, so does $k \mapsto y(k)$.
In summary, we have shown that for each $(\hat{w},\hat{x},\hat{\chi},\hat{\eta})$, the solution $(w(\cdot),x(\cdot),\chi(\cdot),\eta(\cdot))$ to~\eqref{eq:sys_cl_xy2} with initial condition $(\hat{w},\hat{x},\hat{\chi},\hat{\eta})$ satisfies $\lim_{k \to \infty} y(k) = 0$.

\bibliographystyle{siamplain}
\bibliography{pubs}

\begin{thebibliography}{10}

\bibitem{alsalti2025notes}
{\sc M.~Alsalti, V.~G. Lopez, and M.~A. M{\"u}ller}, {\em Notes on data-driven
  output-feedback control of linear {MIMO} systems}, IEEE Transactions on
  Automatic Control,  (2025).

\bibitem{bosso2025data}
{\sc A.~Bosso, M.~Borghesi, A.~Iannelli, G.~Notarstefano, and A.~R. Teel}, {\em
  Data-driven control of continuous-time {LTI} systems via non-minimal
  realizations}, arXiv preprint arXiv:2505.22505,  (2025).

\bibitem{de2019formulas}
{\sc C.~De~Persis and P.~Tesi}, {\em Formulas for data-driven control:
  Stabilization, optimality, and robustness}, IEEE Transactions on Automatic
  Control, 65 (2019), pp.~909--924.

\bibitem{horn2013matrix}
{\sc R.~A. Horn and C.~R. Johnson}, {\em Matrix Analysis}, Cambridge University
  Press, 2013.

\bibitem{hu2024dd-output-regulation}
{\sc Z.~Hu, C.~{De Persis}, J.~W. Simpson-Porco, and P.~Tesi}, {\em Data-driven
  harmonic output regulation of a class of nonlinear systems}, Systems \&
  Control Letters, 200 (2025), p.~106079.

\bibitem{hu2025enforcing}
{\sc Z.~Hu, C.~De~Persis, and P.~Tesi}, {\em Enforcing contraction via data},
  IEEE Transactions on Automatic Control,  (2025).

\bibitem{isidori2017lectures}
{\sc A.~Isidori}, {\em Lectures in feedback design for multivariable systems},
  Springer, 2017.

\bibitem{kailath1980linear}
{\sc T.~Kailath}, {\em Linear Systems}, Prentice-Hall, 1980.

\bibitem{knobloch2012topics-in-control-theory}
{\sc H.~W. Knobloch, A.~Isidori, and D.~Flockerzi}, {\em Topics in control
  theory}, vol.~22, Birkh{\"a}user, 2012.

\bibitem{lee2025input}
{\sc J.~Lee, N.~H. Jo, H.~Shim, F.~D{\"o}rfler, and J.~Kim}, {\em Input-output
  data-driven representation: Non-minimality and stability}, arXiv preprint
  arXiv:2512.01238,  (2025).

\bibitem{lewis2017remarks}
{\sc A.~D. Lewis}, {\em Remarks on stability of time-varying linear systems},
  IEEE Transactions on Automatic Control, 62 (2017), pp.~6039--6043.

\bibitem{li2024controller}
{\sc L.~Li, A.~Bisoffi, C.~{De Persis}, and N.~Monshizadeh}, {\em Controller
  synthesis from noisy-input noisy-output data}, Automatica, 183 (2026),
  p.~112545.

\bibitem{li2026cooperative}
{\sc Y.~Li, W.~Liu, G.~Wang, L.~Xie, J.~Sun, and J.~Chen}, {\em Data-driven
  cooperative output regulation via output feedback control}, IEEE Transactions
  on Automatic Control,  (2026), pp.~1--8.

\bibitem{liu2025data}
{\sc W.~Liu, Y.~Li, J.~Sun, G.~Wang, K.~You, L.~Xie, and J.~Chen}, {\em
  Data-driven internal model control for output regulation}, IEEE Transactions
  on Cybernetics,  (2025), \url{https://doi.org/10.1109/TCYB.2026.3690605}.
\newblock In press, see also arXiv preprint arXiv:2505.09255.

\bibitem{liuguo2025data}
{\sc Y.~Liu and M.~Guo}, {\em Data-driven nonlinear output regulation via
  data-enforced incremental passivity}, arXiv preprint arXiv:2506.06079,
  (2025).

\bibitem{mao2025one}
{\sc J.~Mao, E.~Williams, T.~Mylvaganam, and G.~Scarciotti}, {\em One equation
  to rule them all -- {P}art {II}: Direct data-driven reduction and
  regulation}, arXiv preprint arXiv:2508.17251,  (2025).

\bibitem{Zhu2024output-regulation}
{\sc L.~Zhu and Z.~Chen}, {\em Data informativity for robust output
  regulation}, IEEE Transactions on Automatic Control, 69 (2024),
  pp.~7075--7080.

\end{thebibliography}

\end{document}